\documentclass[11pt,english,a4paper]{smfart}

\usepackage[latin1]{inputenc} 
\usepackage{vaucanson-g}
\usepackage{graphicx}
\usepackage{ae,amsfonts,euscript,enumerate}
\usepackage[cm]{aeguill}

\renewcommand{\theequation}{\thesection.\arabic{equation}}

\newtheorem{thm}{Theorem}[section]

\newtheorem{defn}{Definition}[section]

\newtheorem{prop}{Proposition}[section]
\newtheorem{lem}{Lemma}[section]
\newtheorem{cor}{Corollary}[section]

\newtheorem{rem}{Remark}[section]

\newcommand{\KK}{K\langle\langle x_1,\ldots,x_n\rangle\rangle}

\newcommand{\KKxy}{K\langle\langle x_1,\ldots,x_n,y_1,\ldots y_n\rangle\rangle}

\newcommand{\KKnn}{K\langle\langle x_1,\ldots,x_{n-1}\rangle\rangle}

\begin{document}
\bibliographystyle{plain}
\title[Diagonalization and Rationalization ]{Diagonalization and Rationalization of algebraic Laurent series}
\author{Boris Adamczewski}
%\thanks{BA would}
\address{
CNRS, Universit\'e de Lyon, Universit\'e Lyon 1\\
Institut Camille Jordan  \\
43 boulevard du 11 novembre 1918 \\
69622 Villeurbanne Cedex, France}
\email{Boris.Adamczewski@math.univ-lyon1.fr}

\author{Jason P.~Bell}
\thanks{The First author was supported by ANR grants Hamot and SubTile. The second author was supported by NSERC grant 31-611456.}

\address{
Department of Mathematics\\
Simon Fraser University\\
Burnaby, BC, Canada\\
 V5A 1S6}

\email{jpb@math.sfu.ca}

%\keywords{Skolem-Mahler-Lech theorem, positive characteristic, Christol's theorem, algebraic power series, linear recurrences, zero sets.}

%\subjclass[2000]{}

%\address{l\\
%Department of Mathematics\\
%Simon Fraser University\\
%Burnaby, BC V5A 1S6\\
%Canada}

\bibliographystyle{plain}
%%%%%%%%%%%%%%%%%%%%%%%%%%%%%%%%%%%%%%%%%%%%%%%%%%%%%%%%%%%%%%%%%%%%%%%%%%%%%%%%
\begin{altabstract}  
Nous d\'emontrons une version quantitative d'un r\'esultat  de Furstenberg \cite{Fur} et Deligne \cite{De} : 
la diagonale d'une s\'erie formelle alg\'ebrique de plusieurs variables \`a coefficients dans un corps de 
caract\'eristique non nulle est une s\'erie formelle alg\'ebrique d'une variable.  
Comme cons\'equence, nous obtenons que, pour tout nombre premier $p$,  la r\'eduction modulo $p$ de la diagonale d'une  
 s\'erie formelle alg\'ebrique de plusieurs variables $f$ \`a coefficients entiers est une s\'erie formelle alg\'ebrique de degr\'e au plus 
 $p^{A}$ et de hauteur au plus $A^2p^{A+1}$,  o\`u $A$ est une constante effective ne d\'ependant que du nombre de variables, 
 du degr\'e de $f$ et de la hauteur de $f$.   Cela r\'epond \`a une question soulev\'ee par Deligne \cite{De}.
\end{altabstract}
%%%%%%%%%%%%%%%%%%%%%%%%%%%%%%%%%
%%%%%%%%%%%%%%%%%%%%%%%%%%%%%%%%%%%%%%
\begin{abstract}  
We prove a quantitative version of a result of Furstenberg \cite{Fur} and Deligne \cite{De} 
stating that the the diagonal of a multivariate algebraic 
power series with coefficients in a field of positive characteristic is algebraic.   
As a consequence, we obtain that for every prime $p$ the reduction modulo $p$ of the diagonal of a 
multivariate algebraic power series $f$ with integer coefficients is an algebraic power series of degree 
at most $p^{A}$ and height at most $A^2p^{A+1}$, where $A$ is an effective constant that only depends on the number of variables, 
the degree of $f$ and the height of $f$.   
This answers a question raised by Deligne \cite{De}.
\end{abstract}
%%%%%%%%%%%%%%%%%%%%%%%%%%%%%%%%%%%%%%%%%%%%%5%%%%%%%%%%%%%%%%%%%%%%%%%%%%%%%%%%%%%%%%%%%%%%%%%%%%%%%%%%%%%%%%%%%%%%%
\maketitle

\tableofcontents
%%%%%%%%%%%%%%%%%%%%%%%%%%%%%%%%%%%%%%%%%
\section{Introduction}

A very rich interplay between arithmetic, geometry, transcendence and combinatorics 
arises in the study of homogeneous linear differential equations and especially of  those that ``come from geometry" and the related study 
of Siegel $G$-functions (see for instance \cite{An89,DGS, KZ,St80,St99,Miw06} for discussions that emphasize these different aspects).  
As an illustration, let us recall a few of the many classical results attached to  the differential equation 
$$
t(t-1)y''(t) + (2t-1)y'(t) + \frac{1}{4}y(t) = 0 \, . 
$$

\smallskip

\begin{itemize}

\item[$\bullet$] This differential equation comes from geometry:   
it is the Picard--Fuchs equation of the Legendre family of elliptic curves $\mathcal E_t$ defined by  the equation 
$y^2 = x(x-1)(x-t)$.

\medskip

\item[$\bullet$] The unique solution (up to multiplication by a scalar) that is holomorphic at the origin is the function 
$$
f_1(t) := \frac{2}{\pi} \int_{0}^{\pi/2} \frac{d\theta}{\sqrt{1-t\sin^2 \theta}}  \, \cdot
$$

\medskip

\item[$\bullet$]  For nonzero algebraic numbers $t$ in the open unit disc, $f_1(t)$ is an elliptic integral and $\pi f_1(t)$ is a period in the sense 
of Kontsevich and Zagier \cite{KZ} which are both known to be 
transcendental  (see for instance the complete survey \cite{Miw08}).  
In particular, the function $f_1$ is a transcendental function over $\mathbb Q(t)$. 

%\medskip
 
%\item[$\bullet$]  The function $f(t)$ is a classical Gauss hypergeometric function, namely $f(t) = \prescript{}{_2}{F}_1^{}(1/2,1/2,1,t)$.

\medskip 

\item[$\bullet$]  The function $f_1$ has the following nice hypergeometric Taylor series expansion:  
$$
f_1(16t) =  \sum_{n= 0}^{+\infty}{2n \choose n}^2 t^n \in \mathbb Z[[t]] \, ,
$$
which corresponds to a classical generating function in enumerative combinatorics (associated for instance with 
 the square lattice walks  that start  and end at origin). 

\end{itemize}

\medskip

A remarkable result is that, by adding variables, we can see $f_1$ as arising in a natural way from a much more elementary function, 
namely a rational function.  Indeed, let us consider the  rational function 
$$
R(x_1,x_2,x_3,x_4) := \frac{2}{2-x_1-x_2} \cdot \frac{2}{2-x_3-x_4} \, .
$$
Then $R$ can be expanded as 
$$
\begin{array}{ll}
R &=\displaystyle  \sum_{(i_1,i_2,i_3,i_4)\in\mathbb N^4} a(i_1,i_2,i_3,i_4) \ x_1^{i_1}x_2^{i_2}x_3^{i_3}x_4^{i_4}\\ \\

&\displaystyle =\sum_{(i_1,i_2,i_3,i_4)\in\mathbb N^4} 2^{-(i_1+i_2+i_3+i_4)}{i_1+i_2\choose i_1} {i_3+i_4 \choose i_3} 
\ x_1^{i_1}x_2^{i_2}x_3^{i_3}x_4^{i_4}\, .
\end{array}
$$
Collecting all the diagonals terms, we easily get that 
$$
\Delta(R):= \sum_{n=0}^{+\infty} a(n,n,n,n)t^n = f_1(t)\, .
$$ 
More formally, given a field $K$ and a multivariate power series 
$$
f(x_1,\ldots,x_n) := \sum_{(i_1,\ldots,i_n)\in\mathbb N^n} a(i_1,\ldots,i_n)x_1^{i_1}\cdots x_n^{i_n}
$$
with coefficients in $K$,  we define the {\it  diagonal} $\Delta(f)$ of $f$  as the one 
variable power series 
$$
\Delta(f)(t) := \sum_{n=0}^{+\infty} a(n,\ldots,n) t^n \in K[[t]]\, .
$$ 
Another classical example which emphasizes the richness of diagonals is the following.  
The power series 
$$
f_2(t) :=  \sum_{n= 0}^{+\infty} \sum_{k=0}^{n} {n\choose k}^2 {n+k \choose k}^2 t^n \in \mathbb Z[[t]]\, ,
$$
is a well-known transcendental $G$-function that appears in Ap\'ery's proof of the irrationality of $\zeta(3)$ (see \cite{Fi}).   
It is also known to satisfies the Picard--Fuchs equation associated with a one-parameter family of $K_3$ surfaces \cite{BP84}.  
Furthermore, a simple computation shows that $f_2$ is the diagonal of the five-variable rational function
$$
 \frac{1}{1-x_1} \cdot \frac{1}{(1-x_2)(1-x_3)(1-x_4)(1-x_5) - x_1x_2x_3}  \in\mathbb Z[[x_1,\ldots,x_5]]\, .
$$

These two examples actually reflect a general phenomenon. 
In the case where $K=\mathbb C$, diagonalization may be nicely visualized thanks to Deligne's formula \cite{De} 
via contour integration over a vanishing cycle.  
%$$
%\Delta(f) := \frac{1}{(2i\pi)^n}\oint_{\vert x_2\vert =\cdots =\vert x_n\vert=\varepsilon  }f(x_1,\ldots,x_n)\  \frac{dx_2\cdots dx_n}{x_2\cdots x_n} \cdot
%$$
Formalizing this in terms of the Gauss--Manin connection and De Rham cohomology groups, and using a deep result of Grothendieck,  
one can prove that the diagonal of any algebraic power series with algebraic coefficients  is a Siegel $G$-function that comes from geometry, 
that is, one which satisfies the Picard--Fuchs type equation associated with some one-parameter family of algebraic varieties \cite{An89,Ch86}.   
As claimed by the  Bombieri--Dwork conjecture, this is a picture expected for all $G$-functions.  
Diagonals of algebraic power series with coefficients in $\overline{\mathbb Q}$ thus appear to be a distinguished class of $G$-functions. 
Originally introduced in the study of Hadamard products \cite{CM}, diagonals have since been studied by many authors and for many different reasons \cite{Ch82,Ch84,Ch86,De,DL,Fur,Li,LvdP,St80,St99}.  

\begin{rem} {\rm The same power series may well 
arise as the diagonal of different rational functions, but it is expected that the 
underlying families of algebraic varieties should be connected in some way, such as {\it via} the existence of some isogenies 
(see the discussion in \cite{Ch86}). For instance, $f_1(t)$ is also the diagonal of the three-variable rational function 
$$\frac{4}{4-(x_1+x_2)(1+x_3)} \, ,$$ 
while $f_2(t)$ is also the diagonal of the six-variables rational function 
$$\frac{1}{(1-x_1x_2)(1-x_3-x_4-x_1x_3x_4)(1-x_5-x_6-x_2x_5x_6)} \, \cdot$$}
\end{rem}

\medskip

When $K$ is a field of positive characteristic, the situation is completely different as shown the following nice result.

\begin{defn} 
\emph{A power series 
$f(x_1,\ldots,x_n) \in K[[x_1,\ldots,x_n]]$ 
is said to be {\it algebraic} if it is algebraic over the field of rational functions $K(x_1,\ldots,x_n)$, 
that is,  if there exist polynomials $A_0,\ldots, A_m\in {K}[x_1,\ldots,x_n]$, not all zero, 
such that
$\sum_{i=0}^{m} A_i(x_1,\ldots,x_n)f(x_1,\ldots,x_n)^i = 0.$ 
The {\it degree} of $f$ is  the minimum of the positive integers $m$ for which such a relation holds.  
The (naive) {\it height} of  $f$ is defined as 
the minimum of the heights of the nonzero polynomials $P(Y)\in K[x_1,\ldots,x_n][Y]$ 
 that vanish at $f$, or equivalently, as the height of the minimal polynomial of $f$. 
 The height of a polynomial $P(Y)\in K[x_1,\ldots,x_n][Y]$ is  
the maximum of the total degrees of its coefficients.  }
\end{defn}

\begin{thm}[Furstenberg--Deligne] \label{theo: FD}
Let $K$ be a field of positive characteristc. Then the diagonal of an algebraic power series in $K[[x_1,\ldots,x_n]]$ is algebraic. 
\end{thm}

 Furstenberg \cite{Fur} first proved the case  where $f$ is a rational power series and Deligne \cite{De} extended this result to algebraic 
 power series by using tools from arithmetic geometry.  
Some elementary proofs have then been worked out by  Denef and Lipshitz \cite{DL}, Harase \cite{Har88}, 
Sharif and Woodcock \cite{SW}.   
 The present work  
is mainly motivated by the following consequence of Theorem \ref{theo: FD}.   
Given a prime number $p$ and a power series $f(x):= \sum_{n=0}^{+\infty} a(n) x^n \in \mathbb Z[[x]]$, we denote by  
$f_{\vert p}$ the reduction of $f$ modulo $p$, that is 
$$
f_{\vert p}(x) := \sum_{n=0}^{+\infty} (a(n)\bmod{p})  x^n \in  \mathbb F_p[[x]] \,.
$$ 
Theorem \ref{theo: FD} implies that if $f(x_1,\ldots,x_n)\in \mathbb Z[[x_1,\ldots,x_n]]$ is  algebraic over $\mathbb Q(x_1,\ldots,x_n)$, 
then 
$\Delta(f)_{\vert p}$ is algebraic over $\mathbb F_p(x)$ for every  prime $p$.  
In particular, both the transcendental functions $f_1$ and $f_2$ previously mentioned 
have the remarkable property to have algebraic reductions modulo $p$ for every prime $p$.   

\medskip

It now becomes very natural to ask how the complexity of the algebraic function $\Delta(f)_{\vert p}$ may increase when $p$ 
run along the primes.  A common way to measure the complexity of an algebraic power series is to estimate its degree and its height. 
Deligne \cite{De} obtained a first result in this direction by proving that if $f(x,y)\in \mathbb Z[[x,y]]$ is algebraic,  
then, for all but finitely many primes $p$, $\Delta(f)_{\vert p}$ is an algebraic  power series of degree at most $Ap^B$, where $A$ and $B$ 
do not depend on $p$ but only on geometric quantities associated with $f$.    
He also suggested that  a similar bound should hold for the diagonal of algebraic power series in $\mathbb Z[[x_1,\ldots,x_n]]$.    
Our main aim is to provide the following answer to the question raised by Deligne.

\begin{thm}\label{theo: modp}
Let $f(x_1,\ldots,x_n)\in \mathbb Z[[x_1,\ldots,x_n]]$ be an algebraic power series with degree at most $d$ and height at most $h$. 
Then there exists an effective constant $A:=A(n,d,h)$ 
depending only on $n$, $d$ and $h$, such that $\Delta(f)_{\vert p}$ has degree at most $p^A$ and height at most 
$A^2p^{A+1}$, for every prime number $p$. 
\end{thm}

Theorem \ref{theo: modp} is derived from
the following quantitative version of the Furstenberg--Deligne theorem.

\begin{thm}\label{theo: main}
Let $K$ be a field of characteristc $p>0$ and let $f$ be an algebraic power series in 
$K[[x_1,\ldots,x_n]]$ of degree at most $d$ and height at most $h$. Then there exists an effective constant $A:=A(n,d,h)$ 
depending only on $n$, $d$ and $h$, such that $\Delta(f)$ has degree at most $p^A$ and height at most 
$A^2p^{A+1}$.
\end{thm}

Note that, given a power series $f\in\mathbb Z[[x]]$, the degree and the height of $f_{\vert p}$ are always at most 
equal to those of $f$. Furthermore, 
diagonalization and reduction modulo $p$ commute, that is 
$\Delta(f)_{\vert p} = \Delta(f_{\vert p})$. This shows that Theorem \ref{theo: main} implies Theorem \ref{theo: modp}.

\begin{rem}\emph{ 
Note that if $f(x)\in \mathbb Z[[x]]$ satisfies a nontrivial 
polynomial relation of the form $P(f) \equiv 0 \bmod p$, then we also have the following nontrivial polynomial relation 
$P^m(f) \equiv 0 \bmod \mathbb Z/ p^m\mathbb Z$. 
Thus if a power series is algebraic modulo $p$ with degree at most $d$ and height at most $h$, it is also algebraic 
modulo $p^m$ for every positive integer $m$ with degree at most $md$ and height at most $mh$. 
Under the assumption of Theorem \ref{theo: modp}, we thus have that $f_{\vert p^m}$ is an algebraic power series of degree at most 
$mp^{A}$ and height at most $mA^2p^{A+1}$ for every prime $p$ and every positive integer $m$. 
}
\end{rem}

Theorem \ref{theo: main} actually implies a more general statement given in Theorem \ref{theo: modpbis} below. 
We recall that a ring $R$ is a
\emph{Jacobson} ring if every prime $\mathfrak P\in {\rm Spec}(R)$ is the intersection of the
maximal ideals above it.  
The general form of the Nullstellensatz states that if $S$
is a Jacobson ring and $R$ is a finitely generated $S$ algebra, then $R$ is a
Jacobson ring and each maximal ideal $\mathfrak M$ in $R$ has the property that
$\mathfrak M':=S\cap \mathfrak M$ is a maximal ideal of $S$ and, moreover, $R/\mathfrak M$ is a finite extension
of $S/\mathfrak M'$ (see \cite[Theorem 4.19]{Eis95}). 
Let $R$ be a finitely generated $\mathbb Z$-algebra 
and let $f(x):= \sum_{n\in\mathbb N} a(n) x^n \in  R[[x]]$.   Since
$\mathbb{Z}$ is a Jacobson ring, we have that $R$ is also a Jacobson ring and 
every maximal ideal $\mathfrak M$ of $R$ has the property that $R/\mathfrak M$ is a finite field.   
In particular, if $\mathfrak M$ is a maximal ideal of $R$ and $f(x)=\sum_{n=0}^{+\infty} a(n) x^n \in R[[x]]$, the power series 
$$
f_{\vert \mathfrak M} := \sum_{n=0}^{+\infty} (a(n) \bmod \mathfrak M)x^n
$$ 
has coefficients in the finite field $R/\mathfrak M$.

\begin{thm}\label{theo: modpbis}
Let $K$ be a field of characteristic $0$ and $f(x_1,\ldots,x_n)\in  K[[x_1,\ldots,x_n]]$ 
be an algebraic power series with degree at most $d$ and height at most $h$. 
Then there exists a finitely generated $\mathbb Z$-algebra $R\subseteq K$ such that $\Delta(f)\in R[[x]]$. 
Furthermore, there exists an explicit constant $A:=A(n,d,h)$ 
depending only on $n$, $d$ and $h$, such that, for every maximal ideal $\mathfrak M$ of $R$,  
$\Delta(f)_{\vert \mathfrak M}$ 
is an algebraic power series of degree at most $p^A$ and height at most 
$A^2p^{A+1}$, where $p$ denotes the characteristic of the finite field $R/\mathfrak M$.   
\end{thm} 

\begin{rem}\emph{  
If $R$ is a finitely generated $\mathbb{Z}$-algebra, then for all but
finitely many primes $p$, the ideal $pR$ is proper and hence there are maximal
ideals $M$ such that $R/M$ is a finite field of characteristic $p$.  This follows
from a result of Roquette \cite{Ro} stating that the units group of a finitely
generated commutative $\mathbb{Z}$-algebra that is a domain is a finitely generated abelian group, while 
distinct primes $p_1,\ldots ,p_k$ generate a free abelian subgroup of $\mathbb{Q}^*$ of rank $k$. It
follows that if $K$ is a field of characteristic $0$ and $f(x_1,\ldots,x_n)\in  K[[x_1,\ldots,x_n]]$ is algebraic then, for almost all primes 
$p$, it makes sense to reduce $\Delta(f)$ modulo $p$ and Theorem \ref{theo: modpbis} applies.  }
\end{rem}

Regarding Theorem \ref{theo: modp}, one may reasonably ask about the strength of an upper bound of type $p^{A}$. 
For instance, it is not too difficult to see that both $f_1$ and $f_2$ have degree at most $p-1$ when reduced modulo $p$. 
Curiously enough, we do not find any trace in the literature of any explicit power series $f(x)\in \mathbb Z[[x]]$ 
known to be the diagonal of an algebraic power series and for which the ratio $\deg(f_{\vert p})/p$ is known to be unbounded.    
In general, it seems non-trivial, given the diagonal of a rational function, to get a lower bound for the degree of its reduction modulo $p$.  
As a companion to Theorem \ref{theo: modp}, we prove the following result that provides explicit examples of diagonals 
of rational functions whose reduction modulo $p$ have rather high degree. 

\begin{thm}\label{thm: pA}
Let $s$ be a positive integer and let
$$R_{s} \ : = \  \sum_{r=6}^{6+s-1}  \frac{1}{1-(x_1+\cdots+ x_r)}  \in \mathbb Z[[x_1,\ldots,x_{6+s-1}]] \, .$$ 
Then  $\Delta(R_s)_{\vert p}$ is an algebraic power series of degree at least $p^{s/2}$ for all but finitely many prime numbers $p$.  

\medskip

In particular, for any positive number $N$,  there exists a diagonal of a rational function with integer coefficients 
whose reduction modulo $p$ has degree at least $p^{N}$ 
for all but finitely many prime numbers $p$.
\end{thm}

This shows that the upper bound obtained in Theorem \ref{theo: modp} is ``qualitatively best possible". 
Of course, we do not claim that the dependence in  $n$, $d$ and $h$ of the huge constant $A(n,d,h)$ that can be extracted from 
the proof of Theorem \ref{theo: modp} is optimal: this is not the case. 

\medskip

The outline of the paper is as follows. In Section \ref{sec: strategy}, we describe the strategy of the proof of Theorems 
\ref{theo: modp}, \ref{theo: main} and \ref{theo: modpbis}.  In Section \ref{sec: Fur}, we explain why we will have to work with fields of multivariate 
Laurent series and not only with ring of multivariate power series.  
Such fields are introduced in Section \ref{sec: laurent} where estimates about height and degree of algebraic Laurent series are obtained.  
Sections \ref{sec: cartier}, \ref{sec: rat} and \ref{sec: main} are devoted to 
the proof of Theorem \ref{theo: main}.  We prove Theorem \ref{thm: pA} in Section \ref{sec: high}.  
We discuss some connections of our results to enumerative combinatorics, automata theory and decision problems in Section \ref{sec: decid}. Finally, 
we remark in Section \ref{sec: alg} that our proof of Theorem \ref{thm: pA} incidentally provides a result about algebraic independence of 
$G$-functions satisfying the so-called Lucas property. The latter result is of independent interest and we plan to return to this question in the future. 
%%%%%%%%%%%%%%%%%%%%%%%%%%%%%%%%%%%%
%%%%%%%%%%%%%%%%%%%%%%%%%%%%%%%%%%%%%%%%%%%%%%
\section{Strategy of proof }\label{sec: strategy}

In this section, we briefly describe the main steps of the proof of Theorem \ref{theo: main}.

\medskip

Throughout this section, we let $p$ be a prime number, we let $K$ be a field  of characteristic $p$, and we let
$$f(x_1,\ldots,x_n) := \displaystyle\ \sum_{(i_1,\ldots,i_n)\in \mathbb{\mathbb N}^n} a(i_1,\ldots,i_n)x_1^{i_1}\cdots x_n^{i_n}$$
be a multivariate formal power series with coefficients in $K$ that is algebraic over the field $K(x_1,\ldots,x_n)$.   
We assume that $f$ has degree at most $d$ and height at most $h$. 
Our goal is to estimate the degree of the diagonal of $f$ with respect to $p$. Note that without loss of generality, 
we can assume that $K$ is a perfect field, which means that the map  $x\mapsto x^p$ is surjective on $K$. 

\medskip

\noindent{\bf Step 1 (Cartier operators).} The first idea is to consider a family of operators from $K[[x_1,\ldots,x_n]]$ 
into itself usually referred to as Cartier operators and which are well-known to be relevant in this framework (see  for instance 
\cite{CKMR,SW,Har88,AdBe}). 
 Let 
$$g(x_1,\ldots,x_n) := \ \sum_{(i_1,\ldots,i_n)\in \mathbb{\mathbb N}^n} b(i_1,\ldots,i_n)x_1^{i_1}\cdots x_n^{i_n}$$  
be an element of $K[[x_1,\ldots,x_n]]$.  
For all ${\bf j}:=(j_1,\ldots,j_n)\in \Sigma_p^n:=\{0,1,\ldots ,p-1\}^n$,  
we define the \emph{Cartier operator} $\Lambda_{\bf j}$ from $K[[x_1,\ldots,x_n]]$ into itself by 
\begin{equation}\label{AB:equation:EJ}
\Lambda_{\bf j}(g)\ := \ \sum_{(i_1,\ldots,i_n)\in \mathbb{N}^n} b(pi_1+j_1,\ldots,pi_n+j_n)^{1/p}x_1^{i_1}\cdots x_n^{i_n} \, .
\end{equation}  
Let us denote by $\Omega_n$, or simply $\Omega$ if there is no risk of confusion,  
the monoid generated by the Cartier operators under composition. 

\medskip

In Section \ref{sec: cartier}, we show that the degree (resp. the height) of 
$\Delta(f)$ can be bounded by $p^{N}$  (resp. $N^2p^{N+1}$) 
if one is able to find a $K$-vector space contained in $K[[x_1,\ldots,x_n]]$ of dimension $N$,  
containing $f$ and  invariant under the action of Cartier operators.  This is the object of Propositions \ref{prop: 1} and \ref{prop: 2}.   
In order to prove Theorem \ref{theo: main}, it will thus be enough to exhibit a $K$-vector space 
$V$ such that the following hold.

\medskip

\begin{itemize}

\item[{\rm (i)}] The power series $f$ belong to $V$. 

\medskip

\item[{\rm (ii)}] The vector space $V$ is invariant under the action of $\Omega$.   

\medskip

\item[{\rm (iii)}] The vector space $V$ has finite dimension $N$ that only depends only on $n$, $d$ and $h$.

\end{itemize}

\begin{rem}{\rm  The more natural way to construct an $\Omega$-invariant 
$K$-vector space containing $f$ is to use Ore's lemma, that is, 
to start with the existence of a relation of the form 
$$
\sum_{k=0}^{m} A_kf^{p^k} = 0 \, ,
$$ 
where the $A_k$'s are polynomials. 
This classical approach is for instance used in \cite{CKMR,SW,Har88,AdBe}.  
Furthermore, it can be made explicit in order to bound the dimension of the $\Omega$-invariant vector space one obtains 
(see for instance \cite{Har88,AdBe}). Unfortunately the bound depends on $p$.  
This attempt to answer Deligne's question can be found in \cite{Har88} where Harase   
proved that there exists a number $A$, 
depending on $n$, $d$ and $h$, such that   $\Delta(f)$ has degree at most 
$p^{p^A}$. }
\end{rem}

\medskip

\noindent{\bf Step 2 (The case of rational functions)}.  In the special case where $f$ is a rational function, 
we are almost done for it is easy to construct a $K$-vector space $V$ satisfying (i), (ii) and (iii). 
Indeed, if $f$ is a rational power series there exist two polynomials $A$ and $B$ in 
$K[x_1,\ldots ,x_n]$, with $B(0,\ldots ,0)=1$, such that $f=A/B$.   
By assumption, we can assume that the total degree of $B$ and $A$ are at most $h$. 
Then  it is not difficult to show (see Section \ref{sec: main}) that the $K$-vector space 
$$V :=\{P(x_1,\ldots , x_n)/B(x_1,\ldots, x_n) \mid {\rm deg}(P)\le h\}\subseteq K(x_1,\ldots ,x_n)$$ 
is closed under application of the Cartier operators. 
Furthermore, $V$ has dimension ${n+h\choose n}$. We thus infer from Propositions \ref{prop: 1} and \ref{prop: 2} 
that $\Delta(f)$ is an algebraic power series of degree at most $p^{n+h\choose n}$ and height at most ${n+h\choose n}^2p^{{n+h\choose n}+1}$. 

\medskip

\noindent{\bf Step 3 (Rationalization).} When $f$ is an algebraic irrational power series, the situation is more subtle.  
We would like to reduce to the easy case where $f$ is rational. To achieve this, the idea is to add more variables.  
Indeed, Denef and Lipshitz \cite{DL}, following the pioneering work of Furstenberg \cite{Fur}, 
showed that every algebraic power series $f$ in $K[[x_1,\ldots,x_n]]$ arises as the diagonal 
of a rational power series $R$ in $2n$ variables. Formally, this means that there exists    
$$
R(x_1,\ldots,x_n,y_1,\ldots,y_n) = \sum_{(i_1,\ldots,i_{2n})\in\mathbb N^{2n} } r(i_1,\ldots,i_{2n}) x_1^{i_1}\cdots x_n^{i_{n}} 
y_1^{i_{n+1}}\cdots y_n^{i_{2n}} 
$$
in $K(x_1,\ldots,x_n,y_1,\ldots,y_n)$, such that 
$$
 \Delta_{1/2}(R) := \sum_{(i_1,\ldots,i_n)\in\mathbb N^n} r(i_1,\ldots,i_n, i_1,\ldots,i_n)x_1^{i_1}\cdots x_n^{i_n} = f
 \, .
 $$ 
 Since $\Delta(R)=\Delta(f)$, we are almost done, as we could now replace $f$ by $R$ and use the trick from step $2$. 
The problem we now have is that this rationalization process is not effective. In particular,  it does not give a bound on the height and degree of the rational function $R$ in terms of the height and the degree of the algebraic power series $f$ we started with. 
In order to establish our main result, we need to give an effective version of this procedure.  
Though our approach differs from that used by Denef and Lipshitz it nevertheless hinges on Furstenberg's original work.  
The main issue of Section \ref{sec: rat} is to prove Theorem \ref{theo:mainrat}, which shows 
that one can explicitly control the height of the rational function $R$ in terms of $n$, $d$ and $h$ only. 

\begin{rem}
{\rm   Actually, we do not exactly obtain an effective version of the theorem of Dened and Lipshitz. What we really prove 
in Section \ref{sec: rat} is that every algebraic power series in $n$ variables arises as the diagonal of a rational function 
in $2n$ variables which does not necessarily belong to the ring of power series 
but to a larger field: the field of multivariate Laurent series.   
Elements of such fields also have a kind of generalized power series expansion (so that we can naturally define their diagonals).  
Note that these fields naturally appear when resolving singularities (see \cite{Sat83}). We introduce them in Section \ref{sec: laurent} and 
we explain why it is necessary to use them in the next section.  
}
\end{rem}

%%%%%%%%%%%%%%%%%%%%%%%%%%%%%%%%%%%%%%%%%%%%%%
\section{Comments on Furstenberg's formula}\label{sec: Fur}

In this section, we explain why we have to work with fields of multivariate Laurent series in order 
to describe our effective procedure for rationalization  of algebraic power series.  

\subsection{The one-variable case}\label{subsec: Fur}
Furstenberg  \cite{Fur} gives a very nice way to express a one-variable algebraic power series as the the diagonals of a two-variable rational 
power series.  Though the intuition for his formula comes from the case where $K=\mathbb C$, it remains true for arbitrary fields.

\begin{prop}[Furstenberg]\label{prop: Fur}
Let $K$ be a field and $f(x)\in K[[x]]$ a formal power series with no constant term.  Let us assume that $P(x,y)\in K[x,y]$ is a 
polynomial such that $P(x,f)=0$ and $\partial{P}/\partial{y}(0,0) \not= 0$. Then the rational function 
$$
R(x,y)  :=  y^2\frac{\partial P}{\partial y}(xy,y)/P(xy,y)
$$
belongs to $K[[x,y]]$ and $\Delta(R)=f$. 
\end{prop}

Let us now give an example where Furstenberg's result does not apply directly. 
The algebraic function $f(x)=x\sqrt{1-x}$ has a Taylor series expansion given by 
$$f(x) = x-x^2/2 - x^3/8+\cdots \in \mathbb Q[[x]] \,.$$ 
We thus have $P(x,f)=0$ where $P(x,y)=y^2-x^2(1-x)$. 
Notice that we cannot invoke Proposition \ref{prop: Fur} as $\partial{P}/\partial{y}$ vanishes at $(0,0)$.
However, there is a natural way to overcome this problem which always works with one-variable power series.   
Let us  write
$f(x)=Q(x)+x^i g(x)$, where $Q(x)$ is a polynomial and $g(x)$ is a power series that vanishes at $x=0$.  
If $i$ is chosen to be the order at $x=0$ of the resultant of $P$ and $\partial{P}/\partial{y}$ with respect to the variable $y$, 
then $g$ satisfies the conditions of Proposition \ref{prop: Fur}.   
In our example, the resultant is (up to a scalar) equal to $x^2(1-x)$ and so the order at $x=0$ is $2$. We thus write
$$f(x)=x-x^2/2+x^2g(x)$$ and we see that $g$ satisfies the polynomial equation
$$4xg(x)^2 +4(2-x)g(x)-x \ = \ 0 \,.$$
Set $P_1(x,y) := 4xy^2+(8-4x)y-x$. As claimed, one can check that the partial derivative of $P_1$ with respect to $y$ 
does not vanish at $(0,0)$.  Applying Furstenberg's result, we obtain that $g$ is the diagonal of the rational function 
$$T(x,y)=y^2 (8xy^2+8-4xy)/(4xy^3+8y-4xy^2-xy) \, .$$  
Note that $T(x,y)$ can be rewritten as
$$\frac{1}{8}\cdot\left(8xy^3+8y -4xy^2 \right)\left(1-xy/2 -xy^2/2-x/8\right)^{-1} \, ,$$
which shows that it can be expanded as a power series. 
Finally, we get that $f$ is the diagonal of the rational power series 
$$R(x,y) := xy-x^2y^2/2 + x^2y^2T(x,y) \in \mathbb Q[[x,y]]\, .$$

\subsection{A two-variable example} 
In the case where $f$ is a multivariate algebraic power series, we would still would like to use Furstenberg's formula,  
but new difficulties appear.  
Let us now consider a two-variable example to see why the previous trick could fail in this case. 
Consider the algebraic power series 
$$
f(x,y) = \sqrt{y^2-3y^3+4yx-12y^2x+4x^2-12yx^2} \, .
$$
It has the following power series expansion: 
$$f(x,y)=y+2x-\frac{3}{2}y^2-3yx +\cdots \in \mathbb Q[[x,y]] \, .$$ 
Clearly, $f$ satisfies the polynomial equation $P(x,y,z)=0$ where
$$P(x,y,z)=z^2-y^2-3y^3+4yx-12y^2x+4x^2-12yx^2 \, .$$
Unfortunately, we cannot invoke the natural extension of Furstenberg's formula in this case as the partial derivative 
of $P$ with respect to $z$ vanishes at the origin.

\begin{rem} {\rm Denef and Lipschitz \cite{DL} get around this problem by noting that, given a field $K$, the ring of 
algebraic power series in $K[[x_1,\ldots,x_n]]$ is the Henselization of the local ring $R=K[x_1,\ldots,x_n]_M$, where $M$ is the maximal ideal 
$(x_1,\ldots,x_n)$.  The Henselization is a direct limit of finite \'etale extensions and hence the ring
$R[f]$, formed by adjoining $f$ to $R$, lies in a finite \'etale extension $B$ of $R$.  Because $B$ is a locally standard \'etale extension, 
it is possible to find a generator $\phi$ for the localization of $B$ at a maximal ideal of $B$ above $M$ such that $\phi$ satisfies 
the conditions needed in order to apply Furstenberg's formula.  Moreover, $f$ lies in $B$ and hence it can be expressed as a 
$R$-linear combination of powers of $\phi$.  
This is enough to express $f$ as the diagonal of a rational power series in $K[[x_1,\ldots,x_{n},y_1,\ldots,y_n]]$. However,   
as previously mentioned, this argument is not effective.}
\end{rem}

 \begin{rem} {\rm All one really needs is to be able to make the ring $K[x,y,z]/(P)$ smooth at the origin.  
While there are supposedly ``effective'' ways of resolving singularities (at least in characteristic zero), they are rather sophisticated 
and it is not clear that these methods apply in positive characteristic.  In the case that one is dealing with a one-dimensional variety, 
it is quite simple, as evidenced by the one-variable example above.  The solution is to reduce to a one-dimensional example, 
by localizing $K[x,y,z]/(P)$ at the set of nonzero polynomials in $x$ and embedding this in the one-dimensional 
ring $K((x))[y,z]/(P)$.  This seems to be the simplest way to make things effective.}
\end{rem}

Let us come back to our example and try to apply the approach outlined in the previous remark.   
Let  $L :=\mathbb Q((x))$ denote the field of Laurent series with rational coefficients and consider $P$ as a polynomial in $L[y,z]$ 
and $f$ as a power series in $L[[y]]$.  
Note that $f$ is no longer zero at $(y,z)=(0,0)$, as it now has constant term $f_0(x) = 2x \in L$.  Thus we must write 
$$f(x,y)= 2x + g(x,y),$$ where $g\in L[[y]]$ vanishes at $y=0$ and satisfies the polynomial equation $Q(x,y,g)=0$, where
$$Q(x,y,z)=z^2+ \frac{z(4xy+8x^2)}{(y+2x)} + y(12x^2-4x+12xy-y+3y^2) \, .$$
Note that in general $f_0$ will be an algebraic power series in one-variable and by the algorithm described in \ref{subsec: Fur} 
it can be effectively written as the diagonal of a two-variable rational power series.  Hence we may restrict our attention to $g$.

\medskip

The partial derivative of $Q$ with respect to $z$ is $2z+(4xy+8x^2)/(y+2x)$. When we set $y$ and $z$ 
equal to zero, we obtain $4x\not= 0$.  We can thus use Furstenberg formula if we work in $L[[y,z]]$.  
Set
$$R(x,y,z) := z^2\frac{\partial{Q}}{\partial{z}}(x,yz,z) /Q(x,yz,z) \in \mathbb Q(x,y,z) \,.$$
We thus have 
$$
R(x,y) = \frac{2z^2(4z+2x) + z(4xyz+ 8x^2)}{z(yz+2x)+4xyz + 8x^2+y(12x^2-4x+12xyz -yz+3y^2z^2)} \, \cdot
$$
Note that $R$ cannot be expanded as a power series and thus does not belong to $\mathbb Q[[x,y,z]]$. However, it turns out that it has 
a generalized power series expansion as an element of the bigger field $\mathbb Q\langle\langle x,y,z \rangle\rangle$ 
(see Section \ref{sec: laurent} for a definition). Furthermore, $g$ turns out to be the diagonal of $R$ in this bigger field.  
Finally, we can show that $f(x,y)$ is the diagonal of the $4$-variable rational function
$$
2xt + \frac{2z^2(4z+2xt) + z(4xyzt+ 8x^2t^2)}{z(yz+2xt)+4xyzt + 8x^2t^2+y(12x^2t^2-4xt+12xyzt -yz+3y^2z^2)}  \, \cdot
$$
Again, this rational function belongs to $\mathbb Q\langle\langle x,y,z,t \rangle\rangle$ but not to $\mathbb Q[[x,y,z,t]]$.

\begin{rem}{\rm 
In general, it may happen that the partial derivative of $Q$ with respect to $z$ vanishes at $(y,z)=(0,0)$. In that case, we compute the 
resultant of $Q$ and $\partial{Q}/\partial{z}$ with respect to the variable $z$.  It will be of the form $y^a S$ for some $S$ that is a unit in $L[[y]]$.  
We can then use the algorithm given above in the one-variable case, and rewrite
$$g(x,y)=yg_0(x) + y^2g_1(x)+\cdots + y^a g_a(x)+y^ah(x,y),$$
where $g_0,\ldots ,g_a$ are one-variable algebraic power series and $h\in L[[y]]$ vanishes at $y=0$.  
Again, we can effectively write $g_0,\ldots ,g_a$ as diagonals using the one-variable argument.  
So we may restrict our attention to $h$. 
In this case, we find, by the same reasoning as in the one-variable case, that $h$ satisfies the conditions of Proposition \ref{prop: Fur} 
and so we can put all the information together to finally express $f$ as a diagonal of a generalized Laurent series in 
$\mathbb Q\langle\langle x,y,z,t\rangle\rangle$. In the general case where $f$ is a multivariate algebraic power series with 
coefficients in an arbitrary field, we argue by induction and use the same ideas combining resultants and 
Furstenberg's formula. }
\end{rem}
%%%%%%%%%%%%%%%%%%%%%%%%%%%%%%%%%%%%%%%%%%%%
\section{Fields of multivariate Laurent series}\label{sec: laurent}

In this section, we introduce fields of Laurent series associated with a vector of indeterminates following 
the presentation of Sathaye \cite{Sat83}.

\medskip

\subsection{Fields of Laurent series} 

Let $K$ be a field.  We first recall that the field of Laurent series associated with the indeterminate $x$ is 
as usual defined by 
$$
K((x)) := \left\{ \sum_{i=i_0}^{+\infty} a(i)x^i \mid n_0\in\mathbb Z \mbox{ and } a(i)\in K\right\} \, .   
$$
We then define recursively the field of \emph{multivariate Laurent series associated with the vector of 
indeterminates} ${\bf x}=(x_1,\ldots,x_n)$ by
$$
K\langle\langle {\bf x} \rangle\rangle := K\langle\langle (x_1,x_2,\ldots,x_{n-1}) \rangle\rangle ((x_n)) \, .
$$

\medskip

Let us give a more concrete description of this field. We first define a pure lexicographic ordering 
$\prec$ on the monomials of the form $x_1^{i_1}\cdots x_n^{i_n}$ with 
$(i_1,\ldots ,i_n)\in \mathbb{Z}^{n}$ by declaring that 
$$
x_1\prec x_2 \prec  \cdots \prec x_d \, .
$$ 
This induces a natural 
order on $\mathbb{Z}^n$, which, by abuse of notation, we denote by $\prec$ so that 
$$
(i_1,\ldots,i_n) \prec (j_1,\ldots,j_n)
$$ if these $n$-tuples are distinct and if the largest index $k$ such that $i_k\not=j_k$ satisfies $i_k < j_k$.  
Then it can be shown that the field $K\langle\langle {\bf x} \rangle\rangle$ can be described as the 
collection of all formal series 
$$f(x_1,\ldots,x_n) = \sum_{(i_1,\ldots ,i_n)\in \mathbb{Z}^n} a(i_1,\ldots ,i_n)x_1^{i_1}\cdots x_n^{i_n}$$
whose support is well-ordered, which means that it contains no infinite decreasing subsequence. 
We recall that the support of $f$ is defined by 
$$
{\rm Supp}(f):= \left\{(i_1,\ldots ,i_n)\in \mathbb{Z}^n \mid a(i_1,\ldots ,i_n)\neq 0 \right\} \, .
$$ 

\begin{defn}\emph{
Note that the valuation on $K[x_1,\ldots ,x_n]$ induced by the prime ideal $(x_n)$ extends to 
a valuation on $\KK$.   We let $\nu_n$ denote this valuation and we let $\KKnn[[x_n]]$ denote 
the subring of $\KK$ consisting of all elements $r$ with $\nu_n(r)\ge 0$.}
\end{defn}

%%%%%%%%%%%%%%%%%%%%
\subsection{Algebraic Laurent series} 

Given an $n$-tuple of natural numbers 
$(i_1,\ldots ,i_n)$ and indeterminates $x_1,\ldots ,x_n$, the degree of the monomial $x_1^{i_1}\cdots x_n^{i_n}$ 
 is the nonnegative integer $i_1+\cdots + i_d$. 
Given a polynomial $P$ in $K[x_1,\ldots,x_n]$,  the degree of $P$, $\deg P$,  is defined as 
the maximum of the degrees 
of the monomials appearing in $P$ with nonzero coefficient.   
A central notion in this paper is that of algebraic multivariate Laurent series, that is element of $\KK$ which are 
algebraic over the field of rational functions $K(x_1,\ldots,x_n)$. 

\begin{defn} {\em
We say that  $f(x_1,\ldots,x_n) \in \KK$ 
is \emph{algebraic} if it is algebraic over the field of rational functions $K(x_1,\ldots,x_n)$, 
that is,  if there exist polynomials $A_0,\ldots, A_m\in {K}[x_1,\ldots,x_n]$, not all zero, 
such that
\begin{displaymath}
\sum_{i=0}^{m} A_i(x_1,\ldots,x_n)f(x_1,\ldots,x_n)^i \ = \ 0\, .
\end{displaymath}
The degree of $f$ is defined as the minimum of the positive integer $m$ for which such a relation holds. }
\end{defn} 

\noindent {\bf\itshape Warning.} We have introduced two different notions: the degree of a polynomial and the degree 
of an algebraic function. Since polynomials are also algebraic functions we have to be careful. 
For instance the polynomial $x^2y^3\in K[x,y]$ has degree $5$ but viewed as an element of $K[[x,y]]$ it is an 
algebraic power series of degree $1$.    
In the sequel, we have tried to avoid this kind of confusion.

\begin{defn} {\em Given a polynomial $P(Y)\in K[x_1,\ldots,x_n][Y]$, we define the height of $P$ as 
the maximum of the degrees of the coefficients of $P$. 
The (naive) height of an algebraic power series 
$$f(x_1,\ldots,x_n)=\sum_{{\bf i} \in \mathbb Z^n} a(i_1,\ldots,i_n) x_1^{i_1}\cdots x_n^{i_n}\in 
\KK$$ 
 is then defined as the height of the minimal polynomial of $f$, or equivalently, 
 as the minimum of the heights of the nonzero polynomials $P(Y)\in K[x_1,\ldots,x_n][Y]$ 
 that vanish at $f$.}
\end{defn}

\begin{rem}\emph{  Since $\KK$ is a field, we see that both the field $A$ of rational functions 
$K(x_1,\ldots ,x_n)$  and the field $B$ of fractions of algebraic power series in 
$K[[x_1,\ldots ,x_n]]$ embed in $\KK$. Under these embeddings, we call the elements 
of $A$ the \emph{rational} Laurent power series and we call $B$ the \emph{algebraic} 
Laurent power series. }
\end{rem}

%%%%%%%%%%%%%%%%%%
\subsection{Estimates about height and degree of algebraic Laurent series}

We now collect a few estimates about height and degree of algebraic Laurent series 
that will be useful for proving our main result.

\begin{lem} Let $m$ and $n$ be natural numbers and let $d_1,\ldots ,d_m$ and $h_1,\ldots ,h_m$ be integers.   
Suppose that $f_1,\ldots , f_m\in \KK$ are $m$ algebraic Laurent power series such that $f_i$ has degree 
at most $d_i$ and height at most $h_i$ for each $i$.  Suppose that $A_1,\ldots ,A_m$ are rational functions in 
$K(x_1,\ldots ,x_n)$ whose numerators and denominators have degrees bounded above by some constant $d$. 
Then the following hold.

\medskip

\begin{enumerate}

\item[{\rm (i)}]
The algebraic Laurent series $A_1f_1+\cdots +A_m f_m$ has degree at most 
$d_1\cdots d_m$ and height at most $m(d_1\cdots d_m)(\max(h_1,\ldots ,h_m)+d)$.

\medskip

\item[{\rm (ii)}] The algebraic Laurent series $f_1\cdots f_m$ has degree at most 
$d_1\cdots d_m$ and height at most $m(d_1\cdots d_m)\max(h_1, \ldots, h_m)$.
\end{enumerate}
\label{lem: bound}
\end{lem}

\begin{proof}Let $V$ denote the $K(x_1,\ldots ,x_n)$-vector space spanned by all monomials
$$\{f_1^{i_1}\cdots f_m^{i_m}~\mid ~0\le i_1< d_1,\ldots ,0\le i_m< d_m\}.$$
Then $V$ is a $K(x_1,\ldots ,x_n)$-algebra that has dimension at most $d_1\cdots d_m$ over $K(x_1,\ldots ,x_m)$.  
Note that $A_1f_1+\cdots +A_m f_m$ induces an endomorphism $\phi$ of $V$ by left multiplication.  

\medskip

Let us suppose that $f_i$ has degree $e_i \leq d_i$ over $K(x_1,\ldots ,x_n)$ with minimal polynomial 
$\sum_{k=0}^{e_i} Q_{i,k} Y^{i}\in K[x_1,\ldots, x_n][Y]$.   
Then the field extension $K_i=K(x_1,\ldots,x_n)(f_i)$ is a 
$K(x_1,\ldots,x_n)$-vector space of dimension $e_i$ and
$\{f_i^k~\mid~0\leq k < e_i\}$ forms a basis of this vector space.

Let $$R = K_1 \otimes K_2 \otimes \cdots \otimes K_m\, ,$$ where the tensor products are taken over
$K(x_1,\ldots ,x_n)$. Then $R$ is a $K(x_1,\ldots,x_n)$-vector space of dimension $e_1\cdots e_m$ and 
$$\left\{ f_1^{j_1}\otimes f_2^{j_2}\otimes \cdots \otimes f_m^{j_m} \mid 0\leq j_i<e_i \mbox{ for } i=1,\ldots ,m\right\}$$ 
forms a basis of this vector space.

\medskip

We note that $R$ is a $K(x_1,\ldots ,x_n)$-algebra, since each $K_i$ is a $K(x_1,\ldots ,x_m)$-algebra.  
We regard $R$ as a module over itself.  We then 
have a surjective $K(x_1,\ldots ,x_n)$-algebra homomorphism $g: R \rightarrow V$ 
given by $a_1\otimes a_2 \otimes \cdots \otimes a_m \mapsto a_1 a_2 \cdots a_m$.     

\medskip
 
Next, we let $g_i$ be the element $1 \otimes \cdots\otimes  f_i \otimes \cdots \otimes 1$ in $R$,
where we put $1$s in every slot except for the $i$-th slot, where we put $f_i$.  
Now, we can lift $\phi$ to an $K(x_1,\ldots ,x_n)$-vector space endomorphism $\Phi$ of $W$ defined by 
$$
\Phi(r) = (A_1 g_1 + \cdots +A_m g_m) \cdot r \, ,
$$
for every $r$ in $R$.  
Note that we have 
\begin{equation}\label{eq: g}
g(\Phi(r)) = \phi(g(r)),
\end{equation}
for all $r\in R$.   
 We infer from (\ref{eq: g}) that the characteristic polynomial $P$ of $\Phi$ will also annihilate 
the endomorphism $\phi$, since $g$ is surjective.

\medskip

 Let $R_i$ be a nonzero polynomial in $K[x_1,\ldots ,x_m]$ of degree at most $d$ with the property that $A_iR_i$ 
 is also a polynomial.  Let $H:=\max\{ h_1,\ldots,h_m\}$. Observe that 
 $(A_1 g_1 + \cdots +A_m g_m) \cdot (f_1^{j_1}\otimes f_2^{j_2}\cdots \otimes f_m^{j_m})$ is 
a $K(x_1, \ldots,x_n)$-linear combination of tensors 
$f_1^{j_1}\otimes f_2^{j_2}\cdots \otimes f_m^{j_m}$ with numerators and denominators
of degrees bounded by $H+d$ and each denominator dividing $Q_{i,e_i}R_i$ for some $i$ 
(note that by definition $Q_{i,e_i}$ is nonzero).   
Thus, using a common denominator $Q_{1,e_1}\cdots Q_{m,e_m} R_1\cdots R_m$, 
we see that the endomorphism $\Phi$ can be represented by a $e_1\cdots e_m\times e_1\cdots e_m$ 
matrix whose entries are rational functions with a common
denominator $Q_{1,e_1}\cdots Q_{m,e_m} R_1\cdots R_m$ and with numerators of degree at most $m(H+d)$. 

 This gives that  $P$, the characteristic polynomial of $\Phi$, has height at most $(e_1\cdots e_m)(H+d)m$ 
 (and of course degree equal to $e_1\cdots e_m$).  It follows that $A_1f_1+\cdots +A_mf_m$ has height at most 
 $m(d_1\cdots d_m)(H+d)$ and degree at most $d_1\cdots d_m$, as required. 
 
 \medskip 
 
A similar argument in which we lift $\psi$, the endomorphism given by left multiplication by $f_1\ldots f_m$, gives that $f_1\cdots f_m$ has degree at most $d_1\cdots d_m$ 
and height at most $m(d_1\cdots d_m)H$. 
\end{proof}

%%%%%%%%%%
\begin{lem} 
Let $f\in \KK$ be algebraic of degree $d$ and height at most $h$.  Then
$\vert \nu_n(f)\vert \leq  h$ and $x_n^{-\nu_n(f)}f$ is algebraic of degree $d$ 
and height at most $h(d+1)$.
\label{lem: nu} 
\end{lem}

\begin{proof}
Let $j:=\nu_n(f)$.  By assumption $f$ satisfies a non-trivial polynomial equation of the form 
$P_d f^d +\cdots+ P_1f+P_0 = 0$ with $P_dP_0\neq 0$ and with the degrees of $P_0,\ldots ,P_d$ 
bounded above by $H$.  Then $g:=x_n^{-j}f$ satisfies a polynomial equation of the form
\begin{equation}
\label{eq: PS}
P_d x_n^{-jd} g^d +\cdots  + P_1 x_n^{-j}g + P_0=0\, .
\end{equation}  
If $\vert j\vert >h$ then $\{\nu_n(P_i x_n^{-j i}g^i)~\mid~0\le i\le d\}$ are all distinct, and hence 
Equation (\ref{eq: PS}) cannot hold, a contradiction.  By multiplying by an appropriate power of $f$, 
we see that the height of $x_n^{-j}f$ is at most $dj+h\le h(d+1)$.
\end{proof}

\medskip

Note that, by definition of the field $\KK$, every $f\in \KK$ has a unique infinite decompostion 
\begin{equation}\label{eq: decomposition}
f = \sum_{i=\nu_n(f)}^{+\infty} f_i x_n^{i} \, ,
\end{equation}
where $f_i \in \KKnn$ and $f_{\nu_n(f)}\neq 0$.  For every nonnegative integer $k$, we set 
\begin{equation}
\label{eq: gk}
g_{\nu_n(f)+k} := x_n^{-\nu(f)-k}\left(f- \sum_{i=\nu_n(f)}^{\nu(f) + k} f_i x_n^{i} \right)\, .
\end{equation}
For every integer $r\geq \nu_n(f)$, we thus have the following decomposition: 
\begin{equation}\label{eq: decomposition2}
f = \sum_{i=\nu_n(f)}^{r} f_i x_n^{i} + x_n^r g_r\, .
\end{equation}

\medskip

\begin{lem} Let $d$ and $h$ be natural numbers and let $f$ be an algebraic element of $\KK$ 
of degree at most $d$ and height at most $h$ over $K(x_1,\ldots ,x_n)$.   
Then, for every nonnegative integer $k$, the following hold.

\medskip

\begin{enumerate}

\item[{\rm (i)}] The Laurent series $f_{\nu_n(f)+k}$ is algebraic over $K(x_1,\ldots x_{n-1})$ of degree at most 
$d^{2^k}$ and height at most $8^{k+1}d^{2^{k+2}}h$.

\medskip

\item[{\rm (ii)}] The Laurent series $g_{\nu_n(f)+k}$ is algebraic over $K(x_1,\ldots ,x_n)$ 
of degree at most $d^{2^{k+1}}$ and height at most $8^{k+1}d^{3\cdot 2^{k+1}}h$. 

\end{enumerate}
\label{lem: weirdbounds}
\end{lem}
\begin{proof} 

We prove this by induction on $k$.  We first assume that $k=0$.  
Set $j:=\nu_n(f)$ and  $g=x_n^{-j}f$.   
By Lemma \ref{lem: nu}, $g$ satisfies a polynomial equation of the form
$$
\sum_{i=0}^d Q_i g^i \ = \ 0
$$ 
where the $Q_i$ have degree at most $h(d+1)$ and such that 
$\nu_n(Q_i)=0$ some $i$.   Let us denote by $\phi$ the canonical homomorphism 
from $\KKnn[[x_n]]$ to $\KKnn$ given by $x_n\mapsto 0$.  Then 
$$
0 = \phi\left(\sum_{i=0}^d Q_ig^i\right)  = \sum_{i=1}^d \phi(Q_i)f_{j}^{i} 
$$ and so $f_j$ is an algebraic Laurent series over $K(x_1,\ldots,x_{n-1})$ of degree at 
most $d$ and height at most $h(d+1)$.  
Also by Equation (\ref{eq: gk}) we have that $g_j=fx_n^{-j}- f_j$.  
Lemma \ref{lem: bound} implies that $g_j$ is algebraic over $K(x_1,\ldots,x_n)$ with degree at most $d^2$ 
and height at most $2d^2(h(d+1)+1) \le 8d^6h$.   This establishes the case $k=0$.

Now suppose that the claim is true for all natural numbers less than $k$, for some nonnegative 
integer $k$.  
Let us first note that $f_{\nu_n(f)+k+1} = \phi(g_{\nu_n(f)+k}/x_n)$. Using Lemma \ref{lem: nu}, 
we get that 
\begin{equation}\label{eq: fk1}
\deg f_{\nu_n(f)+k+1} \leq \deg g_{\nu_n(f)+k}
\end{equation}
and
\begin{equation}\label{eq: fk2}
h(f_{\nu_n(f)+k+1}) \leq (\deg g_{\nu_n(f)+k}+1)h(g_{\nu_n(f)+k}) \,. 
\end{equation}
From these relations we deduce by induction that $f_{\nu_n(f)+k+1}$ has degree at most $d^{2^{k+1}}$ 
and height at most $(d^{2^{k+1}}+1)(8^{k+1}d^{3\cdot 2^{k+1}}h) \leq 8^{k+2}d^{2^{k+3}}h$, as required. 

On the other hand, we have the relation 
$$
g_{\nu_n(f)+k+1} = x_n^{-1}g_{\nu_n(f)+k} - f_{\nu_n(f)+k+1} \,.
$$
By Lemma \ref{lem: bound} and Inequalities (\ref{eq: fk1}) and (\ref{eq: fk2}), we obtain that 
\begin{equation}\label{eq: gk1}
\deg g_{\nu_n(f)+k+1} \leq \deg g_{\nu_n(f)+k} \cdot \deg f_{\nu_n(f)+k+1} \leq (\deg g_{\nu_n(f)+k})^2
\end{equation}
and 
$$
h(g_{\nu_n(f)+k+1}) \leq  2(\deg g_{\nu_n(f)+k}^2) \cdot (\max(h(f_{\nu_n+k+1}), h(g_{\nu_n(f)+k}))+1) \, ,
$$
which gives 
\begin{equation}\label{eq: gk2}
h(g_{\nu_n(f)+k+1}) \leq 8 (\deg g_{\nu_n(f)+k})^3 h(g_{\nu_n(f)+k}) \, .
\end{equation} 
From Equations (\ref{eq: gk1}) and (\ref{eq: gk2}), it follows directly by induction that 
$g_{\nu_n(f)+k}$ has degree at most $d^{2^{k+1}}$ and height at most 
$8^{k+1}d^{3\cdot 2^{k+1}}h$, as required.   
This ends the proof. 
\end{proof}

%%%%%%%%%%%%%%%%%%%%%%%%%%%%%%%%%%%%%%%
\section{Cartier operators and diagonals}\label{sec: cartier}

Throughout this section $K$ will denote a perfect field of positive characteristic $p$. 
We recall that a field $K$ of characteristic $p$ is perfect if the map  
$x\mapsto x^p$ is surjective on $K$. 
We introduce a family of operators from $\KK$ into itself, usually referred to as Cartier operators. 
With these operators is associated a Frobenius-type decompositon given by Equation (\ref{eq:fs}) and which is 
well-known to be relevant in this framework. We show that it is possible to bound the degree of the diagonal 
$\Delta(f)$ of a Laurent series $f$ in $\KK$ by finding a finite-dimensional $K$-vector space contained in $\KK$, 
containing $f$ and  invariant under the action of Cartier operators.

\medskip

 Let 
$$
f(x_1,\ldots,x_n) := \ \sum_{(i_1,\ldots,i_n)\in \mathbb{\mathbb Z}^n} a(i_1,\ldots,i_n)x_1^{i_1}\cdots x_n^{i_n}\in 
\KK\, .
$$ 
For all ${\bf j}:=(j_1,\ldots,j_n)\in \Sigma_p^n:=\{0,1,\ldots ,p-1\}^n$,  
we define the \emph{Cartier operator} $\Lambda_{\bf j}$ from $K\langle\langle{x_1,\ldots, x_n}\rangle\rangle$ into itself by 
\begin{equation}\label{AB:equation:EJ}
\Lambda_{\bf j}(f)\ := \ \sum_{(i_1,\ldots,i_n)\in \mathbb{Z}^n} a(pi_1+j_1,\ldots,pi_n+j_n)^{1/p}x_1^{i_1}\cdots x_n^{i_n} \, .
\end{equation}
Note that the support of $\Lambda_{\bf j}(f)$ is well-ordered and thus 
$\Lambda_{\bf j}(f)\in \KK$.  
 We have the following useful decomposition: 
\begin{equation}\label{eq:fs}
f  = \sum_{{\bf j}\in \Sigma_p^n}  \Lambda_{\bf j}( f)^p \ x_1^{j_1}\cdots x_n^{j_n} \, .
\end{equation}

Let us denote by $\Omega_n$, or simply $\Omega$ if there is no risk of confusion,  
the monoid generated by the Cartier operators under composition. We then prove the following result.

\begin{prop} Let $W$ be a $K$-vector space of dimension $d$ included in   
$\KK$ and invariant under the action of $\Omega$. Then for every 
$f\in W$, the Laurent series $\Delta(f) \in K((x))$ is algebraic over $K(x)$ with degree at most 
$p^d$.\label{prop: 1}
\end{prop}

\begin{proof} Let $f(x_1,\ldots,x_n)$ be a Laurent series in $W$. 
Let us first remark that $\Delta$ is a $K$-linear operator. The set 
$$
\Delta(W) := \left\{ \Delta(g) \mid g \in W\right\} 
$$ 
is thus a  $K$-vector subspace of $K[[x]]$ whose dimension is at most equal to $d$.  
Let $r$, $1\leq r \le d$, denote the dimension of this vector space.  Let 
$f_1(x),\ldots ,f_r(x)\in \Delta(W)$ be a basis of $K(x)\otimes_{K}\Delta(W)$, the 
 $K(x)$-vector space generated by the elements of $\Delta(W)$. 

For every integer $i\in \{0,\ldots ,p-1\}$ and every $g(x_1,\ldots ,x_n)\in \KK$, 
we have that 
$$
\Lambda_i ( \Delta(g)) = \Delta(\Lambda_{(i,\ldots , i)}(g)) \, .
$$  
Thus $\Delta(W)$ is invariant under the action of $\Omega_1$. 
Using,  for $i \in\{1,\ldots, r\}$, the Frobenius decomposition 
$$
f_i(x)  =  \sum_{\ell=0}^{p-1}  t^{\ell} \Lambda_\ell( f_j(x))^p \, ,
$$
we get that
$$
f_i(x)  = \sum_{j=1}^r \sum_{\ell=0}^{p-1}  r_{i,j,\ell}(x) f_j(x)^p x^{\ell},
$$
for some rational functions 
$r_{i,j,\ell}(x)$ in $K(x)$. 
There thus exist rational functions $R_{i,j}(x)$, $(i,j)\in \{1,\ldots ,r\}^2$, so that  
$$f_i(x)=\sum_{j=1}^r R_{i,j}(x)f_j(x)^p\, .$$

We claim that the matrix  
$$
M:=\left( R_{i,j}(x)\right)
$$ 
belongs to ${\rm GL}_r(K(x))$. Indeed, if  $M$ were not invertible there would exist  
a nonzero vector $(T_1(x),\ldots ,T_r(x))\in K(x)^r$ such that 
$$
(T_1(x),\ldots ,T_r(x))M = 0\, .
$$ 
This would imply the relation 
$$
\sum_{i=1}^r T_i f_i=0 
$$ 
and we would have a contradiction for  $f_1,\ldots ,f_r$ is a basis of the vector space  
$K(x)\otimes_{K} \Delta(W)$.
By inverting $M$, we immediately obtain that for every integer $i$, $1\leq i \leq r$, 
the Laurent series $f_i^p$ belongs to the $K(x)$-vector space generated by $f_1,\ldots,f_r$.  
Then the field 
$$
L:=K(x)(f_1,\ldots ,f_r)
$$ 
is a finite dimensional $K(x)$-vector space spanned by the elements of the set   
$$
\left\{f_1^{i_1}\cdots f_r^{i_r}~\mid~0\le i_1,\ldots ,i_r<p\right \}\, .
$$  
In particular we have that $[L:K(x)]\le p^r\le p^d$.  
Now, since $f\in W$, we have $\Delta(f) \in L$ and thus  
$[K(x)(\Delta(f)):K(x)]\le p^d$.  This ends the proof. 
\end{proof}

We complete Proposition \ref{prop: 1} by showing that, under the previous assumptions, it is also possible to bound the height 
of the Laurent series $\Delta(f)$.

\begin{prop} Let $W$ be a $K$-vector space of dimension $d$ included in   
$\KK$ and invariant under the action of $\Omega$. Then for every 
$f\in W$, the Laurent series $\Delta(f) \in K((x))$ is algebraic over $K(x)$ with height at most 
$d^2p^{d+1}$.\label{prop: 2}
\end{prop}

\begin{proof} Let $f\in W$. We already showed in the proof of Proposition \ref{prop: 1} that there is a  $K$-vector space $\Delta(W)$ 
of dimension $r\leq d$ that is invariant under the action of the Cartier operators and that contains $\Delta(f)$.  
Let$\{g_1,\ldots , g_r\}$ be a $K$-basis of $\Delta(W)$ with $g_1=\Delta(f)$. 
Then we have
$$\Lambda_i(g_j(x)) = \sum_{k=1}^r c_{i,j}^{(k)} g_k(x)$$ for some constants $c_{i,j}^{(k)}$ in $K$.
Furthermore, we have that 
$$g_j(x) = \sum_{i=0}^{p-1} x^i  \Lambda_i(g_j)(x^p)\, .$$ 
We thus see that each
$g_j(x)$ can be expressed as a polynomial-linear combination of $g_1(x^p),\ldots ,g_r(x^p)$ in which the polynomials 
have degrees uniformly bounded by $p-1$. In other words, there is a matrix-valued function $A(x)\in M_r(K[x])$, in which 
every entry has degree at most $p-1$, such that
$${\bf v}(x) = A(x){\bf v}(x^p),$$ where
${\bf v}(x) :=[g_1(x),\ldots ,g_r(x)]^T$.  
Note that $\det(A(x))$ is nonzero, since if it were, we would have a nonzero row vector ${\bf w}(x)$ such that
${\bf w}(x)A(x)=0$, which would give $w(x){\bf v}(x)=0$, contradicting the independence of $g_1,\ldots ,g_r$.
Then $B(x)=A(x)^{-1}$ is a matrix whose entries are rational functions.  Using the formula for the inverse in terms of minors, 
we see that the entries of $B(x)$ are each of the form $c(x)/d(x)$, where $c(x)$ is a polynomial of degree at most $(r-1)(p-1)$ and 
$d(x)$ is the determinant of $A(x)$, which is a nonzero polynomial of degree at most $r(p-1)$.

By induction, we see that
$${\bf v}(x^{p^{s+1}})=B(x^{p^s})\cdots B(x){\bf v}(x) \, ,$$
 for every nonnegative integer $s$.  
Now let $e_i$ denote the $r\times 1$ column vector whose $i$th coordinate is $1$ and whose other coordinates are $0$.
Then $e_1^T, e_1^T B(x),\ldots ,e_1^T B(x^{p^r})\cdots B(x)$ are $r+1$ row vectors of length $r$ with rational function coordinates.  
Moreover, all the coordinates can be written over the common denominator $d(x)\cdots d(x^{p^r})$, which has degree at most $r(p-1)+r(p-1)p+\cdots +r(p-1)p^r<rp^{r+1}$, and the numerators all have degrees bounded by $rp^{r+1}$. 
It follows that they are linearly dependent, and by clearing the denominator and applying Lemma \ref{lem: x}, 
we see that they satisfy a non-trivial dependence relation 
$$\sum_{i=0}^r Q_i(x) e_1^T B(x^{p^i})\cdots B(x) \ = \ 0 \,,$$ in which $Q_0,\ldots ,Q_r$
are polynomials whose degrees are all bounded by $r^2p^{r+1}\le d^2p^{d+1}$. 
This gives
$$\sum_{i=0}^r Q_i(x) e_1^T B(x^{p^i})\cdots B(x){\bf v}(x) \ = \ 0 \, ,$$
and hence
$$\sum_{i=0}^r Q_i(x) g_1(x^{p^i}) \ = \ 0 \,.$$
 Since $g_1=\Delta(f)$, we get the required bound on the height. 
 \end{proof}
 
\begin{lem} \label{lem: x} Let $K$ be a field and let $r$ and $H$ be natural number.  
Suppose that ${\bf w}_0,\ldots ,{\bf w}_{r}$ are $r+1$ row vectors in $K[x]^r$ whose coordinates 
all have degree at most $H$.  Then there exists a non-trivial dependence relation 
$$\sum_{i=0}^{r} Q_i(x) {\bf w}_i \ = \ 0,$$ in which $Q_0(x),\ldots ,Q_r(x)$
are polynomials in $K[x]$ whose degrees are all bounded by $Hr$. 
\end{lem}

\begin{proof} Let $j$ be the largest natural number for which ${\bf w}_0,\ldots  ,{\bf w}_j$ are 
linearly independent.  Then $j<r$.   Since some $(j+1)\times (j+1)$ minor of the $(j+1)\times r$ 
matrix whose $i$th row is ${\bf w}_i$ is nonzero, it is no loss of generality to assume that the 
``truncated'' row vectors obtained by taking the first $j+1$ coordinates of each of ${\
 w}_0,\ldots ,{\bf w}_j$ are linearly independent.  Let ${\bf w}_0',\ldots ,{\bf w}_r'$ denote the truncated vectors 
of length $j+1$.

Using Cramer's rule, we see that there is a non-trivial solution $[P_0,\ldots ,P_{j+1}]$ to the vector equation
$$\sum_{i=0}^{j+1} P_i(x) {\bf w}_i' \ = \ 0,$$ in which $P_{j+1}=-1$ and for $k\le j$, $P_k$ is given by a ratio of 
two $j\times j$ determinants; the denominator is the determinant of the $(j+1)\times (j+1)$ matrix whose $i$th row is ${\bf w}_i'$ 
and the numerator is the determinant of the $(j+1)\times (j+1)$ matrix whose $i$th row is ${\bf w}_i'$ unless $i=k$, 
in which case the row is given by ${\bf w}_{j+1}'$.  We note that the degrees of the numerators and of the common 
denominator are all bounded by $H(j+1)\le Hr$.  By clearing the common denominator, we get a polynomial solution
$$\sum_{i=0}^{j+1} Q_i(x) {\bf w}_i' \ = \ 0,$$ in which the $Q_i$ have degrees uniformly bounded by $Hr$.
By construction, 
$$\sum_{i=0}^{j+1} Q_i(x) {\bf w}_i \ = \ 0,$$ and the result now follows by taking $Q_k(x)=0$ for $j+2\le k\le r$.  
\end{proof}

%%%%%%%%%%%%%%%%%%%%%%%%%%%%%%%%%%%%%%%%%%%%%%
\section{Rationalization of algebraic Laurent series}\label{sec: rat}

Throughout this section, $K$ will denote an arbitrary field. It is known that every algebraic power series in  
$K[[x_1,\ldots,x_n]]$ arises as the diagonal of a rational power series  in $2n$ variables. 
This result is due to  Denef and Lipshitz \cite{DL} 
who used an idea of Furstenberg \cite{Fur}. It is inefective in the sense that it does not say  how large the height of the rational function can be with respect to the height and degree of the algebraic power series we start with. 
In order to establish our main result, we need to prove an effective version of this rationalization process 
for algebraic Laurent series in $\KK$. Though our approach differs from the one of Denef and Lipshitz it is 
also based on Furstenberg's pioneering work.

\medskip

Let us first recall some notation. 
Given a Laurent series in $2n$ variables 
$$f(x_1,\ldots ,x_{2n}):=\sum_{(i_1,\ldots,i_{2n})\in\mathbb Z^{2n}} 
a(i_1,\ldots,i_{2n})x_1^{i_1}\cdots x_{2n}^{i_{2n}}
 \in K\langle\langle x_1,\ldots,x_{2n} \rangle\rangle \, ,
 $$
we define the diagonal operator $\Delta_{1/2}$ from $K\langle\langle x_1,\ldots,x_{2n} \rangle\rangle$ into  
 $\KK$ by 
 $$
 \Delta_{1/2}(f) := \sum_{(i_1,\ldots,i_n)\in\mathbb Z^n} a(i_1,\ldots,i_n, i_1,\ldots,i_n)x_1^{i_1}\cdots x_n^{i_n}
 \, .
 $$ 

We are now ready to state the main result of this section. 

\begin{thm}\label{theo:mainrat} Let $d$ and $h$ be natural numbers and let $f(x_1,\ldots,x_n) \in \KK$ 
be an algebraic Laurent series of degree $d$ and height at most $h$.   
Then there is an explicit number $N(n,d,h)$ depending only on $n$, $d$ and $h$, and a rational Laurent power series  
$R \in \KKxy$ of height at most $N(n,d,h)$ such that 
$f=\Delta_{1/2}(R)$. 
\end{thm}

In order to prove Theorem \ref{theo:mainrat}, we will prove three auxiliary results. 
We also need to introduce a third type of diagonal operators. 
Given a Laurent series in $n+1$ variables 
$$
f(x_1,\ldots ,x_n,y) := \sum_{ (i_1,\ldots,i_{n+1} ) \in\mathbb Z^{n+1} } 
a(i_1, \ldots, i_{n+1} ) x_1^{i_1}\cdots x_{n}^{i_n}y^{i_{n+1}}
 $$
 in $K\langle\langle x_1,\ldots,x_n,y \rangle\rangle$,  
we define the diagonal operator $\Delta_{x_n,y}$ from $K\langle\langle x_1,\ldots,x_n,y \rangle\rangle$ 
into  $\KK$ by 
 $$
 \Delta_{x_n,y}(f) := \sum_{(i_1,\ldots,i_n)\in\mathbb Z^n} a(i_1,\ldots,i_n, i_n)x_1^{i_1}\cdots x_n^{i_n}
 \, .
 $$ 

Our first auxiliary result is essentially (an efective version of) Furstenberg's lemma. 
We include a proof here for completeness. 

\begin{lem} \label{lem: Furst}
Let $P(x_n,y) \in K\langle\langle x_1,\ldots,x_{n-1}\rangle\rangle [x_n,y]$ be a polynomial 
of degree at most $d$ in $x_n$ and $y$ and suppose that all coefficients appearing in $P$ are algebraic elements of 
$K\langle\langle x_1,\ldots,x_{n-1}\rangle\rangle$ of degree at most $d_0$ and height at most $h$.  
Suppose that $f(x_1,\ldots ,x_n)\in K\langle\langle x_1,\ldots,x_n \rangle\rangle$  
is a solution to the equation $P(x_n,f)=0$ with $\nu_n(f)\geq 1$ and that $\partial{P}/\partial{y}$ does not belong 
to the ideal $(x_n,y)K\langle\langle x_1,\ldots,x_{n-1}\rangle\rangle[x_n,y]$.    
Then 
$$
R(x_n,y):=
y^2 \frac{\partial{P}}{\partial{y}}(x_ny,y)/P(x_ny,y)
$$ 
is a rational function of the variables $x_n$ and $y$ that belongs to 
$K\langle\langle x_1,\ldots,x_{n-1}\rangle\rangle[[x_n,y]]$ and such that 
$$
\Delta_{x_n,y}\left(R)\right)=f \, .
$$ 
Furthermore,  the numerator and denominator of $R$ 
have total degree at most  $2d+1$ in $x_n$ and $y$ and their coefficients are algebraic elements of 
$K\langle\langle x_1,\ldots,x_{n-1}\rangle\rangle$ of degree at most $d_0$ and height at most $h$.  
\end{lem}

\begin{proof} We first observe that our assumptions on $P$ ensure that 
both polynomials $y^2 \partial{P}/\partial{y}(x_ny,y)$ and $P(x_ny,y)$ have total degree at most 
$2d+ 1$ in $x_n$ and $y$. It also implies that the coefficients of these polynomials are algebraic 
elements of $K\langle\langle x_1,\ldots, x_{n-1}\rangle\rangle$ 
of degree at most $d_0$ and height at most $h$.

Since $y=f(x_1,\ldots ,x_n)$ is a solution to the equation
$P(x_n,y)=0$, we see that 
\begin{equation}
P(x_n,y)=(y-f(x_1,\ldots ,x_n))Q(x_1,\ldots ,x_n,y) \, ,
\label{eq: 1} 
\end{equation}
for some nonzero polynomial 
$Q\in K\langle\langle x_1,\ldots x_{n-1}\rangle\rangle[[x_n]][y]$.  
Differentiating (\ref{eq: 1}) with respect to $y$, we get that 
\begin{equation}\label{eq: 2} 
\frac{\partial{P}}{\partial{y}} =  Q + (y - f)\frac{\partial{Q}}{\partial{y}} \, \cdot
\end{equation}
Looking modulo the prime ideal $(t_d,y)$ and using the fact that $\partial{P}/\partial{y}$ 
does not vanish at $(x_n,y)=(0,0)$, we infer from (\ref{eq: 2}) that  
$Q$ is nonzero when evaluated at $(x_n,y)=(0,0)$.  Hence $Q$ is a unit in the 
power series ring $K\langle\langle x_1,\ldots x_{n-1}\rangle\rangle[[x_n,y]]$.  
Then Equations (\ref{eq: 1}) and (\ref{eq: 2}) give
$$
\frac{1}{P}\cdot \frac{\partial{P}}{\partial{y}} = 
\frac{1}{y-f}  + \frac{1}{Q} \cdot \frac{\partial{Q}}{\partial{y}}  \, \cdot
$$
Thus
\begin{eqnarray*}
R(x_n,y) 
& = & \frac{y^2}{ P(x_1,\ldots ,x_{n-1},x_ny,y) } \cdot \frac{\partial{P}}{\partial{y}}(x_1,\ldots ,x_{n-1},x_ny,y) \\
& = & \frac{y^2}{y-f(x_1,\ldots ,x_{n-1},x_ny)} + \\
& ~ & \frac{y^2}{Q(x_1,\ldots ,x_{n-1},x_ny,y)} \cdot  \frac{\partial{Q}}{\partial{y}}(x_1,\ldots ,x_{n-1},x_ny,y) \\
& = & \frac{y}{1-y^{-1}f(x_1,\ldots ,x_{n-1},x_ny)} +  \\ 
& ~ &  \frac{y^2}{Q(x_1,\ldots ,x_{n-1},x_ny,y)} \cdot  \frac{\partial{Q}}{\partial{y}}(x_1,\ldots ,x_{n-1},x_ny,y) \, .
\end{eqnarray*}

Furthermore, since $\nu_n(f)\geq 1$,  
we have that $1-y^{-1}f(x_1,\ldots ,x_{n-1},x_ny)$ is a unit in the ring 
$K\langle\langle x_1,\ldots x_{n-1}\rangle\rangle[[x_n,y]]$. We thus obtain that 
$R\in K\langle\langle x_1,\ldots x_{n-1}\rangle\rangle[[x_n,y]]$.  
On the other hand, since $1/Q \cdot \partial{Q}/\partial{y}$ belongs to 
$K\langle\langle x_1,\ldots x_{n-1}\rangle\rangle[[x_n,y]]$, 
we clearly have 
$$
\Delta_{x_n,y}\left(y^2  \frac{\partial{Q}}{\partial{y}}(x_1,\ldots ,x_{n-1},x_ny,y)
/Q(x_1,\ldots ,x_{n-1},x_ny,y)\right) = 0 \,.
$$
Thus
\begin{eqnarray*}
\Delta_{x_n,y}(R)
& = & \Delta_{x_n,y} \left(\frac{y}{1-y^{-1}f(x_1,\ldots ,x_{n-1},x_ny)}\right)\\
& = & \sum_{j\ge 0}\Delta_{x_n,y}\left( y^{1-j}f(x_1,\ldots ,x_{n-1},x_ny)\right) \\
& = & \Delta_{x_n,y}\left( f(x_1,\ldots ,x_{n-1},x_ny)\right) \\
& = & f(x_1,\ldots ,x_n) \, .
\end{eqnarray*}
This ends the proof. 
\end{proof}

%%%%%%

In the previous result, we make the assumption that our algebraic Laurent series $f$ is a 
root of a polynomial $P$ such that  $\partial{P}/\partial{y}$ does not belong to the ideal 
$(x_n,y)K\langle\langle x_1,\ldots,x_{n-1}\rangle\rangle[x_n,y]$. We now remove this assumption.

\begin{lem} Let $P(x_1,\ldots ,x_n,y)\in K[x_1,\ldots ,x_n][y]$ be a nonzero polynomial of degree 
$d\geq 2$ in $y$ and suppose that the coefficients of $P$ (in $K[x_1,\ldots ,x_n]$) all have degree at most 
$h$.  Let $f$ be an algebraic Laurent series in $\KK$. 
Suppose that $\nu_n(f)=1$ and that $f$ satisfies the polynomial equation $P(x_1,\ldots ,x_n,f)=0$.   
Then there is a rational function $R \in K\langle\langle x_1,\ldots x_{n-1}\rangle\rangle (x_n,y)$ such 
that the following hold. 
 
\begin{itemize} 

\medskip

\item[(i)] One has $f = \Delta_{x_n,y}(R)$ \, .

\medskip

\item[(ii)] The numerator and denominator of $R$ have total degrees at most 
$h(2d-1)(2d+1)+2h+1$ in $x_n$ and $y$.

\medskip

\item[(iii)] The coefficients of $R$ are algebraic Laurent series in  
$K\langle\langle x_1,\ldots x_{n-1}\rangle\rangle$ 
of degree at most  $d^{2^{h(2d-1)}-1}$  
and height at most 
$d^{8dh2^{h(2d-1)}}$.

\end{itemize}
\label{lem: diag}
\end{lem}

%%%%%

\begin{proof} 
We first note that if $\partial{P}/\partial{y}$ does not belong to the ideal $(x_n,y)K\langle\langle x_1,\ldots,x_{n-1}\rangle\rangle[x_n,y]$, 
then the result follows directly from Lemma \ref{lem: Furst} (with much better bounds). We thus assume from now on that 
$\partial{P}/\partial{y}\in (x_n,y)K\langle\langle x_1,\ldots,x_{n-1}\rangle\rangle[x_n,y]$. 

By assumption $\nu_n(f)=1$.  
We thus infer from Equation (\ref{eq: decomposition2}) that 
 for every positive integer  $r$, we have the following decomposition:
 $$
 f=  f_1x_n+\cdots+f_rx_n^r + x_n^rg_r \, ,
 $$
 where $f_1,\ldots,f_r$ are algebraic Laurent series in $\KKnn$ and $g_r$ is an algebraic Laurent series in $\KK$ 
 with $\nu_n(g_r)\geq 1$.  Our aim is now to prove that, for a suitable $r$,  $g_r$ 
does satisfy a polynomial relation as in Lemma \ref{lem: Furst}.  
 
Let $S(x_1,\ldots , x_n)\in K[x_1,\ldots x_n]$ denote the resultant with respect to the variable $y$ of the 
polynomials $P$ and $\partial{P}/\partial{y}$.  
Since by assumption $P\in (x_n,y_n)K[x_1,\ldots ,x_n,y_n]$ 
and $\partial{P}/\partial{y}\in (x_n,y_n)\KKnn[x_n,y_n]$,  
we get that  $S\in x_nK[x_1,\ldots ,x_n]$.  There thus exists a  polynomial 
$T\in K[x_1,\ldots ,x_n]$ with $\nu_n(T)=0$ and a positive integer $r$ such that $S=x_n^rT$.  
Furthermore, using the determinantal formula for the resultant, we obtain that 
$S$ is the determinant of a $(2d-1)\times (2d- 1)$ matrix whose entries are polynomials in 
$K[x_1,\ldots ,x_n]$ of degree at most $h$. It follows that 
\begin{equation}
r\leq h(2d-1) \,.
\label{eq: majr}
\end{equation} 

Set $V(x_1,\ldots,x_n) := f_1x_n + \cdots + f_rx_n^r$.  
Denoting by $A_i(x_1,\ldots,x_n)$ the coefficients of $P$, we get that 
$$
\sum_{i=0}^d A_i(x_1,\ldots ,x_n) (V + x_n^r g_r)^i =0 \, .
$$ 
Setting  
\begin{equation}
B_i(x_1,\ldots,x_n):= \frac{1}{i!} \cdot \frac{\partial^i{P}}{\partial{y_n^i}}(x_1,\ldots ,x_n,V(x_1,\ldots,x_n)) \, ,
\label{eq: Bi}
\end{equation}
we easily check that 
\begin{equation}
\label{eq: 12}
\sum_{i=0}^d B_i(x_1,\ldots,x_n) (x_n^rg_r)^i = 0\, .
\end{equation}
Moreover, we note that each $B_i$ is a polynomial in $x_n$ of degree at most $(d-i)r+h$ whose coefficients are 
algebraic Laurent series in $\KKnn$. 

On the other hand, since $S$ is the resultant of $P$ and $\partial{P}/\partial{y}$ with respect to the variable $y$, 
there exist two polynomials 
$A(x_1,\ldots , x_n,y)$ and $B(x_1,\ldots , x_n,y)$ in $K[x_1,\ldots,x_n,y]$ such that 
\begin{equation}
\label{eq: 11} 
S=AP+B\frac{\partial{P}}{\partial{y}} \, \cdot
\end{equation}  
Note that 
\begin{eqnarray*} P(x_1,\ldots , x_n,V)& =& P(x_1,\ldots ,x_n,V)-P(x_1,\ldots ,x_n,f) \\
&=& (V-f)C(x_1,\ldots, x_n)\\
&=&
x_n^rg_rC(x_1,\ldots, x_n) 
\end{eqnarray*} 
for some $C\in \KKnn[[x_n]]$.  Thus $x_n^{r+1}$ divides 
$$
P(x_1,\ldots ,x_n,V)
$$ 
in $\KKnn[x_n]$.  Substituting $y=V$ into Equation (\ref{eq: 11}) gives 
\begin{eqnarray*}
x_n^rT(x_1,\ldots ,x_n)-A(x_1,\ldots ,x_n,V)P(x_1,\ldots ,(x_n,V)& \\ 
=  B(x_1,\ldots ,x_n,V)\frac{\partial{P}}{\partial{y_n}}(x_1,\ldots ,x_n,V)\, .
\end{eqnarray*}
It follows that  
$$
\nu_n(B_1) = \nu_n\left(\frac{\partial{P}}{\partial{y}}(x_1,\ldots,x_n,V)\right) = r 
$$ 
and thus  $\nu_n(B_k x_n^{rk})\geq 2r$ for every $k$, $1\leq k\leq d$.  
Then Equation (\ref{eq: 12}) implies that $\nu_n(B_0)\geq 2r+1$ since $\nu_n(g_r)\geq 1$. 
In particular, for every integer $k$, $0\leq k\leq d$, 
the quantity $C_k := B_kx_n^{rk}/x_n^{2r}$  belongs to $\KKnn[x_n]$. 
Setting 
$$
Q(x_1,\ldots,x_n,y):= \sum_{i=0}^{d} C_i(x_1,\ldots,x_n)y^{i}\, ,
$$ 
we obtain that $Q$ is a polynomial in $x_n$ and $y$ whose coefficients are algebraic Laurent series in $\KKnn$ and such 
that $Q(x_1,\ldots,x_n,g_r)=0$. 
Furthermore, since $\nu_n(g_r)\geq 1$ and $\nu_n(C_1)=0$, we have that $\partial{Q}/\partial{y}$ does not belong 
to the ideal $(x_n,y)K\langle\langle x_1,\ldots,x_{n-1}\rangle\rangle[x_n,y]$. 
It follows that the pair $(Q,g_r)$ satisfies the assumption of Lemma \ref{lem: Furst}. 

\medskip

In order to apply Lemma \ref{lem: Furst}, it just remains to estimate the degree of $Q$ in $x_n$ and $y$ and also 
the height and the degree of the coefficients of $Q$ (as algebraic Laurent series in $\KKnn$).   

First, an easy computation using (\ref{eq: Bi}) and the definition of $V$ gives that the degree of $Q$ in $x_n$ and $y$ is at most 
\begin{equation}\label{eq: d1}
d_1 := dr+h\, .
\end{equation} 
 
On the other hand, we infer from Lemma \ref{lem: weirdbounds}  that each $f_i$ is an algebraic Laurent series in $\KKnn$ with 
degree at most $d^{2^{i-1}}$ and height at most $8^{i}d^{2^{i+1}}h$. Then Lemma \ref{lem: bound}  implies that 
 $V(x_1,\ldots,x_n)$ is a polynomial in $x_n$ whose coefficients are  algebraic Laurent series in $\KKnn$ with 
degree at most $d^{2^{r}-1}$ and height at most $rd^{2^{r}-1}(8^{r}d^{2^{r+1}}h+r)\leq 2r8^rd^{2^{r+2}}h$. 
We also note that for every $k$, $1\leq k\leq d$, $V^k$ is a polynomial in $x_n$ whose coefficients are algebraic Laurent series in 
$\KKnn$ with degree at most 
$d^{2^{r}-1}$, while Lemma \ref{lem: bound} implies that the height of these coefficients 
is at most $k(d^{2^{r}-1})^k2r8^rd^{2^{r+2}}h$.  Furthermore, the definition of $B_k$ implies that 
$$
B_k = \sum_{j=0}^{d-k} {j+k \choose k} A_{j+k}(x_1,\ldots,,x_n)V^j \, , 
$$
for every integer $k$, $0\leq k\leq d$. 
Lemma \ref{lem: bound} thus gives  that $B_k$  is a polynomial in $x_n$ whose 
coefficients are algebraic Laurent series in $\KKnn$ with degree at most $d^{2^{r}-1}$ and height at most 
$$
(d-k+1)(d^{2^{r}-1})^{d-k+1}(k(d^{2^{r}-1})^k2r8^rd^{2^{r+2}}h+h) \, .
$$ 
It follows finally that $C_k$ is a polynomial in $x_n$ whose coefficients  
are algebraic Laurent series in $\KKnn$ of degree at most 
\begin{equation}\label{eq: d2}
d_2 := d^{2^{r}-1}
\end{equation}
 and height at most 
 \begin{equation}\label{eq: h2}
h_2 := (d+1)(d^{2^{r}-1})^{d+1}(d(d^{2^{r}-1})^d2r8^rd^{2^{r+2}}h+h) \, .
\end{equation}

\medskip

We now infer from Lemma \ref{lem: Furst} that there is a rational function $U \in \KKnn(x_n,y_n)$ 
whose numerator and denominator have total degree at most 
$d_3 := 2d_1 + 1$  
in $x_n$ and $y_n$ such that  $\Delta_{x_n,y}(U)= g_r$. 
Moreover, the coefficients of $U$ are algebraic Laurent series in $\KKnn$ of 
degree at most 
$d_2$ and height at most $h_2$. 

Then setting $R :=(x_ny)^r U(x_1,\ldots,x_n,y)+ V(x_1,\ldots,x_{n-1},x_ny)$, we easily obtain that 
$$
\Delta_{x_n,y}(R) = f \, .
$$ 
Furthermore, $R$ is a rational function in the variables $x_n$ and $y$
whose numerator and denominator have degree at most 
$$
d_3 + r \leq  h(2d-1)(2d+1)+2h+1 \,.
$$  
This proves (i) and (ii). 

It thus remains to prove that (iii) holds. 
Note that by construction all coefficients of $R$ are algebraic Laurent series in $\KKnn$ that belong to the field extension 
$K(x_1,\ldots,x_{n-1})(f_1,\ldots,f_r)$. It follows that they all have degree at most $d^{2^r-1}$ (the product of the degree of each $f_i$). 
Furthermore, we infer from (\ref{eq: majr}) that $ d^{2^r-1} \leq d^{2^{h(2d-1)}-1}$, as required. 

Note also that the denominator of 
 $(x_ny_n)^r U(x_n,y_n)+V(x_ny_n)$ can be chosen to be the same as that of $U$.  The numerator, however, is a sum of at 
 most $r+1$ elements, each of which is of the form $fg$, where $f$ and $g$ are algebraic elements in $\KKnn$ 
 with degree at most $d_2$ and height at most  $h_2$. 
Lemma \ref{lem: bound} gives that each product has height at most $2d_2^2h_2$. 
Applying Lemma \ref{lem: bound} again, we obtain that the coefficients appearing 
in $R$ are algebraic elements of $\KKnn$ whose heights are all bounded by
\begin{eqnarray*}
h_3 := (r+1)d_2^{r+1}(2d_2^2h_2+1)\, .
\end{eqnarray*}
A simple computation using (\ref{eq: majr}), (\ref{eq: d2}) and (\ref{eq: h2}), shows that 
$h_3 \leq d^{8dh2^{h(2d-1)}}$, 
as required.
This ends the proof.  
\end{proof}

\begin{lem}\label{lem: XX}
Let $M, d$ and $h$ be positive integers and let $f(x_n,y)$ be a rational function in $\KKnn(x_n,y)$ defined by 
$$
f(x_n,y) := \frac{\displaystyle\sum_{0\le i,j< M} \alpha_{i,j} x_n^i y^j}{\displaystyle\sum_{0\le i,j< M} \beta_{i,j} x_n^i y^j} \, ,
$$ where each $\alpha_{i,j}$ and $\beta_{i,j}$ are algebraic Laurent series in $\KKnn$ with degree at most $d$ and height at most $h$.  
Then there are two polynomials 
$$A(x_n,y)=\sum_{0\leq i,j\le (M-1)d^{M^2} } \gamma_{i,j} x_n^i y^j \in \KKnn[x_n,y]$$ and 
$B(x_1,\ldots,x_n,y)\in K[x_1,\ldots,x_n,y]$ such that the following conditions hold. 
\begin{itemize}

\medskip

\item[{\rm (i)}] $f = A/B$.

\medskip

\item[{\rm (ii)}] The polynomial $B$ has total degree at most $(h+2M-2) d^{M^2}$. 

\medskip

\item[{\rm (iii)}] Each $\gamma_{i,j}$ is an algebraic Laurent series in the field $\KKnn$ with degree at most $d^{(d+1)M^2}$  
and height at most $M^{2(d^{M^2})} d^{(d+1)M^2 \cdot M^{2\left(d^{M^2}\right)} } d^{M^2} d^{(d^{M^2})} h$. 
\end{itemize}
\end{lem}

\begin{proof}
For every $(i,j)\in \{0,\ldots,M-1\}^2$, we denote by $P_{i,j}(X) \in K[x_1,\ldots,x_{n-1}][X]$ the minimal polynomial of $\beta_{i,j}$ 
over $K(x_1,\ldots,x_{n-1})$. By assumption, $P_{i,j}$ has degree $d_{i,j}\leq d$ in $X$ and its coefficients have total degree at most $h$. 
The polynomial $P_{i,j}$ can be split into linear factors in an algebraic closure, say $L$, of $K(x_1,\ldots,x_{n-1})$.    
There thus exist a polynomial $C_{i,j}(x_1,\ldots,x_{n-1}) \in K[x_1,\ldots,x_{n-1}]$ and (possibly equal) algebraic elements 
$\beta_{i,j}^{(1)} = \beta_{i,j}, \beta_{i,j}^{(2)},\ldots,\beta_{i,j}^{(d_{i,j})}$ in $L$ such that 
\begin{equation}\label{eq:pij}
P_{i,j}(X) = C_{i,j}(x_1,\ldots,x_{n-1}) \prod_{k=1}^{d_{i,j}} (X - \beta_{i,j}^{(k)}) \, .
\end{equation}
Let us define the polynomial 
$$
C(x_1,\ldots,x_{n-1}) := \prod_{0\leq i,j < M}  C_{i,j}(x_1,\ldots,x_{n-1}) 
$$
and the set 
$$
{\mathcal S} := \prod_{0\le i,j< M}\left\{1,\ldots, d_{i,j}\right\}  \, .
$$
We thus note that   
\begin{equation}\label{eq:polyB}
B(x_1,\ldots,x_n,y) := C \cdot \prod_{(k_{0,0},\ldots,k_{M-1,M-1})\in {\mathcal S} }
 \left( \sum_{0\le i,j< M} \beta_{i,j}^{(k_{i,j})} x_n^i y^j\right) 
\end{equation}
belongs to $K[x_1,\ldots ,x_n,y]$. Indeed, by construction, for every $(i,j)\in \{0,\ldots,M-1\}^2$, $B$ is a symmetric 
polynomial in the $\beta_{i,j}^{(k)}$, $1\leq k \leq d_{i,j}$, and thus the result follows from Equation (\ref{eq:pij}). 

\medskip

Now, set  
\begin{equation}\label{eq:polyA}
A(x_n,y)  := \left(\frac{B(x_1,\ldots,x_n,y)}{\displaystyle\sum_{0\le i,j< M} \beta_{i,j} x_n^i y^j} \right) 
 \left(\sum_{0\le i,j< M} \alpha_{i,j} x_n^i y^j\right) \, .
\end{equation}
The assumptions made on the $\beta_{i,j}$ and the $\gamma_{i,j}$, and the definition of $B$ ensure that $A$ belongs to $\KKnn[x_n,y]$ 
and has total degree at most $(M-1)d^{M^2}$ in $x_n$ and $y$. This shows the existence of algebraic Laurent series $\gamma_{i,j}\in \KKnn$ 
such that $$\displaystyle A(x_n,y)=\sum_{0\leq i,j\le (M-1)d^{M^2} } \gamma_{i,j} x_n^i y^j .$$  
Furthermore, the definition of $A$ implies that $f=A/B$. 
Thus (i) is satisfied. 

\medskip

We infer from (\ref{eq:polyB}) that $B$ has total degree at most $(M-1)d^{M^2}$ in $x_n$ 
and $y$.  Also, for each $i$ and $j$, the coefficient of $x_n^i y^j$ in $B$ is a polynomial of 
degree at most $d^{M^2}$ in the coefficients of $P_{i,j}(X)$ and 
hence has total degree at most $h \cdot d^{M^2}$ in $x_1,\ldots ,x_{n-1}$.  We deduce that 
$B$ has total degree at most $(h + 2M-2)d^{M^2}$, which proves (ii).   

 \medskip 
  
Now, let $E$ denote the field extension of $K(x_1,\ldots,x_{n-1})$ formed by adjoining all the $\alpha_{i,j}$ and all the $\beta_{i,j}^{(k)}$. 
Then $[E:K(x_1,\ldots,x_{n-1})]\leq d^{(d+1)M^2}$. By definition of $A$, the coefficients $\gamma_{i,j}$ all  belong to $E$ and are thus all 
algebraic Laurent series of degree at most $d^{(d+1)M^2}$. Furthermore, 
it follows from (\ref{eq:polyA}) that each $\gamma_{i,j}$ can be obtained as a sum of at most $M^{2d^{M^2}}$ 
algebraic elements, each of which is a product of $d^{M^2}$ algebraic elements of degree at most $d$ and height at most $h$.   
Using Lemma \ref{lem: bound}, we get that the $\gamma_{i,j}$ are all algebraic Laurent series over $K(x_1,\ldots ,x_{n-1})$ of 
 height at most $M^{2(d^{M^2})} d^{(d+1)M^2\cdot M^{2\left(d^{M^2}\right)}} d^{M^2} d^{(d^{M^2})} h$. This proves (iii) and concludes the proof.
\end{proof}

%%%%%%%%%%%%%%%%%%%%%%%%%%%%%%%%%%%%%%%%%%%
%\section{Upper bound for the height}

%%%%%%%%%%%%%%%%%%%%%%%%ye provto

We are now ready to prove the main result of this section. 

\begin{proof}[Proof of Theorem \ref{theo:mainrat}] 
We prove this by induction on $n$.  
Let $f(x_1,\ldots,x_n)\in \KK$ be an algebraic Laurent series of degree at most $d$ and height at most $h$.  

We first infer from Lemma \ref{lem: nu} that $\nu_n(f)\leq h$. Furthermore, arguing as for the proof of Lemma \ref{lem: nu}, we get that 
$\widetilde{f} := x_n^{-\nu_n(f)+1}f$ is an algebraic Laurent series of degree at most $d$ and height at most $hd$.  
We also note that by definition $\nu_n(\widetilde{f})\geq 1$.  

\medskip

Let us prove the case where $n=1$. 
By Lemma \ref{lem: diag}, there exist a rational function $R(x_1,y)\in K(x_1,y)$ whose height is at most $hd(2d-1)(2d+1)+2hd+1$ and such that 
 $\Delta_{1/2}(R)=\widetilde{f}$. Thus, $S(x_1,y):= (x_1y)^{\nu_n(f)-1}R(x_1,y)$ is 	a rational function whose height is at most 
 $2h+ hd(2d-1)(2d+1)+2hd+1$
 and such that $\Delta_{1/2}(S)=f$. This proves the case $n=1$ with $N(1,d,h) := hd(2d-1)(2d+1)+2h(d+1)+1$.

\medskip

Let us assume now that $n\geq 2$ is a fixed integer and that the conclusion of the theorem holds for every natural number less than $n$. 
By Lemma \ref{lem: diag}, there exists a rational function 
$R(x_n,y_n)\in \KKnn(x_n,y_n)$ such that $\widetilde{f}=\Delta_{x_n,y_n}(B)$. Furthermore, 
the numerator and the denominator of $R$ have degree at most $$M:=hd(2d-1)(2d+1)+2hd+1$$  
 in $x_n$ and $y_n$ and the coefficients of $R$ are algebraic Laurent power series in $\KKnn$ of degree at most  
 $$d_0 :=  d^{2^{hd(2d-1)}-1}$$ and height at most 
$$
h_0 := d^{8d^2h2^{hd(2d-1)}} \, .
$$
By Lemma \ref{lem: XX}, we may write $R(x_n,y_n)$ as $A(x_n,y_n)/B(x_1,x_2,\ldots ,x_n,y_n)$ where $B$ 
is a polynomial in $K[x_1,\ldots ,x_n,y_n]$ of total degree at most $(h_0+ 2M-2)d_0^{M^2}$ and 
$A\in \KKnn[x_n,y_n]$ is 
a polynomial of degree at most $M_1:=(M-1)d_0^{M^2}$ and whose coefficients are all algebraic over 
$K(x_1,\ldots ,x_{n-1})$ of degree at most $d_1:=d_0^{(d_0+1)M^2}$ and height at most 
$h_1:=M^{2(d_0^{M^2})} d_0^{(d_0+1)M^2 \cdot M^{2\left(d_0^{M^2}\right)} } d_0^{M^2} d_0^{(d_0^{M^2})} h_0$.

\medskip

We can thus write
$$A=\sum_{0\leq i,j\le M_1} \gamma_{i,j}x_n^i y_n^j \, .$$  Then by the inductive hypothesis, each $\gamma_{i,j}$ 
is the diagonal of some rational function $$R_{i,j}(x_1,\ldots ,x_{n-1},y_1,\ldots ,y_{n-1})$$
in $\KKnn$ whose height is at most $N(n-1,d_1,h_1)$.  

Consider the rational function $S(x_1,\ldots ,x_n,y_1,\ldots ,y_n)\in \KK$ defined by
$$S:= \frac{\displaystyle\sum_{0\leq i,j\le M_1} R_{i,j}(x_1,\ldots ,x_{n-1},y_1,\ldots ,y_{n-1})x_n^iy_n^j}
{ B(x_1y_1,\ldots ,x_{n-1}y_{n-1}, x_n,y_n)} \, \cdot
$$

Then by construction, we have that $\Delta_{1/2}(S) = \widetilde{f}$. Taking $T=(x_ny_n)^{\nu_n(f)-1} S$, we obtain that $\Delta_{1/2}(T) = f$.  
By noting that $\vert \nu_n(f)\vert \le h$, we see that $T$ has 
height at most  $2(h-1)+ N(n-1,d_1,h_1)$.
This concludes the proof by taking $N(n,d,h) := 2(h-1)+N(n-1,d_1,h_1)$. 
\end{proof}

%%%%%%%%%%%%%%%%%%%%%%%%%%%%%%%%%%%%%%%

\section{Proof of Theorems \ref{theo: main} and \ref{theo: modpbis}}\label{sec: main}

We are now ready to conclude the proof of our main result. 
What we will actually prove is the following natural extension of Theorem \ref{theo: main}
 to fields of multivariate Laurent series.
 
 \begin{thm}\label{theo: mainbis}
Let $K$ be a field of characteristc $p>0$ and $f$ be an algebraic power series in 
$\KK$ of degree at most $d$ and height at most $h$. Then there exists an explicit constant $A:=A(n,d,h)$ 
depending only on $n$, $d$ and $h$, such that $\Delta(f)\in K((x))$ is an algebraic Laurent series of degree at most $p^A$ and height at most 
$A^2p^{A+1}$.
\end{thm}

\begin{proof}
Let $K$ be a field of characteristic $p>0$ and let $f(x_1,\ldots,x_n)$ be a power series in $\KK$ of degree at most $d$ and 
height at most $h$. Without loss of generality, we can assume that $K$ is a perfect field (otherwise we just enlarge it and consider the perfect closure of $K$). 
By Theorem \ref{theo:mainrat}, there exist an explicit  positive number $N$ depending only on $n$, $d$, and $h$, and a rational function $R(x_1,\ldots,x_n,y_1,\ldots,y_n)$  in 
$K\langle\langle _1,\ldots,x_n,y_1,\ldots,y_n\rangle\rangle$, with height at most $N$, such that $\Delta_{1/2}(R) = f$.  
The latter property clearly implies that $\Delta(R) =\Delta(f)$. 

We are now going to exhibit a $K$-vector space containing $R$ and invariant under the action of the monoid $\Omega_{2n}$ generated by 
Cartier operators.  Since $R$ is a rational function with height at most $N$, there exist two polynomials $P$ and $Q$ in $K[x_1,\ldots,x_n,y_1,\ldots,y_n]$ 
with total degree at most $N$ and such that $R=P/Q$.  Set 
$$
V := \left \{\frac{S(x_1,\ldots , x_n,y_1,\ldots,y_n)}{Q(x_1,\ldots, x_n,y_1,\ldots,y_n) }
\mid {\rm deg}(S)\le N\right\} \subseteq K(x_1,\ldots ,x_n,y_1,\ldots,y_n) \, .
$$
Note that $V$ is a $K$-vector space of dimension $A:={N+2n \choose N}$. 
Let $S(x_1,\ldots , x_n,y_1,\ldots,y_n)/Q(x_1,\ldots , x_n,y_1,\ldots,y_n)$ be an element 
of $V$ and let ${\bf j}\in \{0,\ldots ,p-1\}^{2n}$. Let $\Lambda_{\bf j}$ denote the Cartier operator associated with ${\bf j}$ (see Section \ref{sec: cartier} for a definition). A useful property of Cartier operators is that $\Lambda_{\bf j}(g^ph)=g\Lambda_{\bf j}(h)$ for every pair $(g,h)\in \KK^2$. 
We thus deduce that 
$$\Lambda_{\bf j}(S/Q)=\Lambda_{\bf j}(SQ^{p-1}/Q^p) = \frac{1}{Q}\cdot \Lambda_{\bf j}(SQ^{p-1}) \, .$$
Since ${\rm deg}(SQ^{p-1})\le pN$, we can write
$SQ^{p-1}$ as
$$
SQ^{p-1}=\sum_{{\bf i} := (i_1,\ldots,i_{2n})\in \{0,\ldots ,p-1\}^{2n}} S_{\bf i}^p \ x_1^{i_1}\cdots x_n^{i_n}y_1^{i_{n+1}}\cdots y_n^{i_{2n}}
$$ where each $S_{\bf i}$ is a polynomial of total degree at most $N$.  
Now, the unicity of such a decomposition ensures that $S_{\bf j} = \Lambda_{\bf j}(SQ^{p-1})$. This implies that 
$$\Lambda_{\bf j}(S/Q)=S_{\bf j}/Q \, .$$  
Thus $\Lambda_{\bf j}(S/Q)$ belongs to $V$, which shows that $V$ is invariant under the action of Cartier operators. 
By Propositions \ref{prop: 1} and \ref{prop: 2}, it follows that 
$\Delta(R)$ is algebraic over $K(x)$ with degree at most $p^{A}$ and height at most $A^2p^{A+1}$.  Since $\Delta(R)=\Delta(f)$, 
this ends the proof. 
\end{proof}

\begin{proof}[Proof of Theorem \ref{theo: modpbis}] 
Theorem \ref{theo: modpbis} is essentially a consequence of Theorem \ref{theo: main} (see the discussion in the introduction about Jacobson rings). 
The  only thing that remains to be proven is that if $K$ is a field of characteristic zero and if $f(x_1,\ldots ,x_n)\in K[[x_1,\ldots,x_n]]$ is algebraic 
then the coefficients of $\Delta(f)$ all belong to a finitely generated $\mathbb Z$-algebra $R\subseteq K$. To see this, we use the result of Denef and Lipshitz \cite{DL} claiming that $f(x_1,\ldots ,x_n)$ can be written as the diagonal of a rational power series in $2n$ variables, say $P(x_1,\ldots , x_{2n})/Q(x_1,\ldots ,x_{2n})\in K(x_1,\ldots ,x_{2n})$.  Without loss of generality, we may 
assume that $Q(0,\ldots ,0)=1$.  Let us write $Q=1-U$, with $U\in (x_1,\ldots ,x_{2n})K[x_1,\ldots ,x_{2n}]$.  Let $R$ denote the finitely 
generated $\mathbb{Z}$-subalgebra of $K$ generated by the coefficients of $P$ and $U$.  Then 
the identity $P/Q=\sum_{k\ge 0} P U^k$ shows that  
all the coefficients of $P/Q$ lie in $R$. Since $\Delta(f)= \Delta(P/Q)$, it follows that all coefficients of 
$\Delta(f)$ also belong to $R$.   
\end{proof}
%%%%%%%%%%%%%%%%%%%%%%%%%%%%%%%%%%%%%%%%%%%%%%%%%%

\section{Diagonals of rational functions with high degree modulo $p$}\label{sec: high}

In this section, we prove Theorem \ref{thm: pA}.  Throughout this section we make use of the following notation already introduced in the 
introduction of this paper:  
if $f(x):=\sum_{n=0}^{\infty} a(n)x^n\in \mathbb Z[[x]]$ and $p$ is a prime number, 
we denote by $f_{\vert p} := \sum_{n=0}^{\infty} (a(n)\bmod p) x^n\in\mathbb F_p[[x]]$ the reduction of $f(x)$ modulo $p$.   
An expression like ``$f$ vanishes modulo $p$" just means that $f_{\vert p}$ is identically equal to zero. 
Also, given two polynomials $A(x)$ and $B(x)$ in 
$\mathbb Z[x]$, the expression  
{``}$A(x)$ divides $B(x)$ modulo $p$" means that $A_{\vert p}(x)$ divides $B_{\vert p}(x)$ in $\mathbb F_p[x]$. 

\medskip

An essential property that will be used all along this section is the so-called \emph{Lucas property}.

\begin{defn}\emph{ We say that a sequence $a:\mathbb{N}\to \mathbb{Z}$ has the \emph{Lucas property} 
if for every prime $p$ we have $a(pn+j)\equiv a(n)a(j)\, (\bmod ~p)$.   We let $\mathcal{L}$ denote the set of all 
power series in $\mathbb{Z}[[x]]$ that have constant coefficient one, whose sequence of coefficients has the Lucas property, and that 
satisfy a homogeneous linear differential equation with coefficients in $\mathbb Q(x)$.  }
\end{defn}

\begin{rem}\label{rem: 1}
\emph{We note that if $f(x)=\sum_{n\geq 0} a(n)x^n \in \mathcal{L}$ and $p$ is a prime number, then 
$$f(x) \equiv A(x)f(x^p)~(\bmod\, p) \,,$$ 
where $A(x) :=\sum_{n=0}^{p-1} a(n)x^n$. Furthermore, since $a(0)=1$, we always have that the polynomial 
$A_{\vert p}(x)$ is not identically zero. In the sequel, there will be no problem with dividing by such polynomial $A(x)$ 
in congruences relation modulo $p$.} 
\end{rem}

\begin{lem}
\label{lem: 0}
Let $f_1,\ldots, f_s\in \mathbb Z[[x]]$ such that $f_{1 \vert P},\ldots, f_{s \vert p}$ are linearly dependent over 
$\mathbb F_p$ for infinitely many prime numbers $p$. Then, $f_1,\ldots, f_s$ are linearly dependent over $\mathbb Q$. 
\end{lem}

\begin{proof}
Let $a_i(n)$ denote the $n$th coefficient of $f_i$. Let us consider 
$$
 \begin{pmatrix} a_1(0)&a_1(1)&a_1(2)&\cdots \\
a_2(0)&a_2(1)&a_2(2)& \cdots\\
\vdots &\vdots &\vdots &\ldots\\
a_r(0)&a_r(1)&a_r(2)&\cdots
\end{pmatrix},
$$
 the $s\times \infty$ matrix whose coefficient in position $(k,n)$ is $a_k(n)$. Given a prime $p$ such that 
$f_{1\vert P},\ldots, f_{s\vert p}$ are linearly dependent over $\mathbb F_p$, we thus have that any $s\times s$ minor has determinant 
that vanishes modulo $p$. Since this holds for infinitely many primes $p$, we obtain that all $s\times s$ minors are equal to zero, 
which implies the linear dependence of $f_1,\ldots, f_s$ over $\mathbb Q$.  
\end{proof}

\begin{lem}
\label{lem: 1} Let $f_1(x),\ldots ,f_s(x)\in 
\mathcal{L}$ and set $f(x) := f_1(x)+\cdots + f_s(x)$.  Let us assume that the following hold. 
\begin{itemize}

\medskip

\item[(i)] The functions $f_1,\ldots ,f_s$ are linearly independent over $\mathbb{Q}$. 

\medskip

\item[(ii)]  The inequality $\deg (f_{\vert p}) < p^{s/2}$ holds for infinitely many prime numbers $p$.

\end{itemize}

\medskip

\noindent Then for infinitely many primes $p$ there is a polynomial $Q(x_1,\ldots, x_s)\in \mathbb{Z}[x][x_1,\ldots ,x_s]$ 
of total degree at most $\sqrt{p}s$ in $x_1,\ldots ,x_d$ such that $Q_{\vert p}$ is nonzero and 
$Q(f_1(x),\ldots ,f_s(x)) \equiv 0~\bmod ~p$.
\end{lem}

\begin{proof} 
Let $\mathcal{P}_0$ denote the infinite set of primes $p$ for which  
$\deg(f_{\vert p})<p^{s/2}$.  From now on, we let $p$ denote a fixed element of $\mathcal P_0$. 
Let 
$\mathcal{S}$ denote the set of all numbers of the form 
$i_0+i_1 p+\cdots + i_s p^{s-1}$ with $0\le i_0,\ldots ,i_{s-1} <\sqrt{p}$.  
Then $|\mathcal{S}|\ge \sqrt{p}^s$ and hence there is a non-trivial relation of the form
\begin{equation}
\label{eq: sum}
\sum_{0\le i_0,\ldots ,i_{s-1}<\sqrt{p}} c_{i_0,\ldots ,i_{s-1}}(x) f(x)^{i_0}\cdots f(x)^{p^{s-1} i_{s-1}} \ \equiv \ 0 ~(\bmod ~p).
\end{equation}
Since by assumption each $f_i$ belongs to $\mathcal L$, we infer from Remark \ref{rem: 1} that 
there are polynomials $A_1,\ldots ,A_s$ of degree at most $p-1$ such that
$f_i(x)\equiv A_i(x)f_i(x^p) \bmod p$.  We also have that each $f_i$ has constant coefficient $1$, which implies that $A_{i}(x) \not\equiv 0~\bmod ~p$.  Then we have that 
$$
f_i(x^{p^j}) \equiv \frac{ f_i(x)}{\displaystyle\prod_{m=1}^{j-1} A_i(x^{p^{m-1}})}~\, (\bmod ~p) 
$$
for every positive integer $j$. 
Letting $B_{i,j}(x) := \displaystyle\prod_{m=1}^{j-1} A_i(x^{p^{m-1}})$ for $j\ge 1$ and $B_{i,j}(x) := 1$ for $j=0$,  
Equation (\ref{eq: sum}) can be rewritten as
\begin{equation}
\label{eq: sum2}
\sum_{0\le i_0,\ldots ,i_{s-1}<\sqrt{p}} c_{i_0,\ldots ,i_{s-1}}(x) \prod_{j=0}^{s-1} 
 \left(\frac{f_1(x)}{B_{1,j}(x)} + 
\cdots + \frac{f_s(x)}{B_{s,j}(x)} \right)^{i_j} \ \equiv\ 0 ~(\bmod ~p).
\end{equation}
If we expand the left-hand side, we obtain the existence of a polynomial 
$P(x,x_1,\ldots, x_s)\in \mathbb{Z}(x)[x_1,\ldots ,x_s]$ 
of  total degree at most $\sqrt{p}s$ in $x_1,\ldots ,x_s$ such that $P(x,f_1,\ldots, f_s)$ vanishes modulo $p$.  
If $P_{\vert p}$ is nonzero, we just have to multiply 
by the common denominator in Equation (\ref{eq: sum2}) (which is nonzero modulo $p$ by Remark \ref{rem: 1}) to 
obtain a nonzero polynomial $Q\in \mathbb{Z}[x][x_1,\ldots ,x_s]$ with the desired properties.  

\medskip

It thus remains to prove that $P$ does not vanishes modulo $p$. From now on, we may assume that $P_{\vert p}$ is identically zero  
and we will show this yields a contradiction.   
Let $y_1,\ldots ,y_s$ be indeterminates.  Then
\begin{equation}
\label{eq: sum3}
\sum_{0\le i_0,\ldots ,i_{s-1}<\sqrt{p}} c_{i_0,\ldots ,i_{s-1}}(x) \prod_{j=0}^{s-1} \left(\frac{y_1}{B_{1,j}(x)} + 
\cdots + \frac{y_s}{B_{s,j}(x)} \right)^{i_j} \ \equiv \ 0 ~(\bmod ~p) \, .
\end{equation}
Let $z_1,\ldots ,z_s$ be new variables defined by
$$z_j:=y_1/B_{1,j}(x) + \cdots +y_s/B_{s,j}(x) \, .$$
Then 
$$\sum_{0\le i_0,\ldots ,i_{s-1}<\sqrt{p}} c_{i_0,\ldots ,i_t}(x)\, z_0^{i_1}\cdots z_{s}^{i_{s-1}} \ \equiv \ (0 ~\bmod ~p)$$ 
and since $c_{i_0,\ldots, i_{s-1}}$ is nonzero modulo $p$ for some $(i_0,\ldots, i_{s-1})$, we see that $x,z_1,\ldots ,z_{s}$ 
are algebraically dependent over $\mathbb F_p(x)$.  

On the other hand, we have 
$$[z_1,\ldots, z_s]^{\rm T}=B^T[y_1,\ldots ,y_s]^T$$ where $B$ is an $s\times s$ matrix whose $(i,j)$-entry is $1/B_{i,j-1}(x)$.  
We claim that $B$ is invertible as a matrix with coefficients in $\mathbb F_p(x)$.  
To see this, let us assume that $B$ is not invertible.  Then there exists a nonzero vector of 
polynomials $[c_1(x),\ldots, c_{s}(x)]\in \mathbb F_p[x]^s$ such that
$$\sum_{j=1}^s c_j(x)/B_{i,j-1}(x) \ \equiv\ 0 (\bmod p)$$ for $i=1,\ldots ,s$.  
But, by construction, $f_i(x)^{p^j} \equiv f_i(x)/B_{i,j}(x)~ (\bmod\, p)$ and hence we must have
$$
\sum_{j=1}^s c_j(x)f_i(x)^{p^{j-1}-1} \ \equiv  0~(\bmod p)
$$ 
for $i=1,\ldots ,s$. Thus   
$$\sum_{j=1}^s c_j(x)f_i(x)^{p^{j-1}} \ \equiv 0~(\bmod p)$$ 
for $i=1,\ldots ,s$. 
In particular any $\mathbb F_p$-linear combination of $
f_1,\ldots ,f_s$, say $y:=\lambda_1f_1+\cdots + \lambda_sf_s$ satisfies the relation 
$$\sum_{j=1}^s c_j(x)y^{p^{j-1}} \ \equiv \ 0\,\bmod p \, .$$  
Regarding this expression as a polynomial in $y$, we get at most $p^{s-1}$ distinct roots in $\mathbb F_p[[x]]$.  
Since there are $p^{s}$ $\mathbb F_p$-linear combinations of the form 
$\lambda_1f_{1}+\cdots + \lambda_sf_{s}$, it follows that at least two different linear combinations must be the same.  Thus  
$\lambda_1f_1+\cdots +\lambda_sf_{s}=0$ for some $\lambda_1,\ldots ,\lambda_s\in \mathbb F_p$ not all of which are zero. 
Since this holds for infinitely many primes $p$, 
Lemma \ref{lem: 0} gives the linear dependence of $f_1,\ldots,f_s$ over $\mathbb Q$, a contradiction with (i).   
This proves that $B$ is invertible. 

\medskip

Now since $B$ is invertible,  we can express $y_1,\ldots ,y_s$ as $\mathbb F_p(x)$-linear combinations of 
$z_1,\ldots ,z_s$ and thus $\mathbb F_p(x,y_1,\ldots,y_s) \subset \mathbb  F_p(x,z_1,\ldots ,z_s)$. This contradicts the fact that 
$x,z_1,\ldots,z_s$ are algebraically dependent over $\mathbb F_p(x)$.  Thus the 
polynomial $P_{\vert p}$ is not identically zero, which ends the proof. 
\end{proof}

\begin{lem} Let $f(x)=\displaystyle\sum_{n= 0}^{+\infty}a(n)x^n\in \mathcal{L}$, let $p$ be a prime number, and set    
$A(x):= \sum_{n=0}^{p-1}a(n)x^n$.   
Then there exist a nonzero polynomial $Q(x)\in\mathbb Z[x]$ 
and a number $m$ (both independent of $p$) such that for every non-constant irreducible factor $C(x)$ of $A_{\vert p}(x)$ 
either $C^m(x)$ does not divide $A_{\vert p}(x)$ or $C(x)$ divides $Q_{\vert p}(x)$.
\label{lem: 2}
\end{lem}

\begin{proof} By assumption $f$ satisfies a relation of the form
$$
\sum_{i=0}^r P_i(x) f^{(i)}(x) \ = \ 0 \, ,
$$
where $P_0,\ldots,P_r$ belong to $\mathbb Z[x]$ and $P_r$ is nonzero.  
Let $d$ be the largest of the degrees of $P_0,\ldots ,P_r$.  
Note that by Remark \ref{rem: 1}, we have $f(x) \equiv A(x)f(x^p)~(\bmod ~p)$ and thus 
$f(x)\equiv A(x)~\bmod ~(x^p, p)$.  This gives:
$$\sum_{i=0}^r P_i(x) A^{(i)}(x) \equiv \ 0~\bmod ~(x^p, p) \, .$$
Thus we may write
\begin{equation}
\label{eq: X}
\sum_{i=0}^r P_i(x)A^{(i)}(x) \equiv x^p B(x)~(\bmod ~p) \, ,
\end{equation} 
for some polynomial $B\in \mathbb F_p[x]$ with $\deg B < d$. Now take $m=r+d$ and suppose that $A_{\vert p}(x)$ has an irreducible 
factor $C(x)\in\mathbb F_p[x]$ such that $C^m$ divides $A_{\vert p}$.  
Then $C(x)^{m-r}$ divides the left-hand side of 
Equation (\ref{eq: X}) modulo $p$ and hence must divide $x^pB(x)$.  
Since by assumption $a(0)=1$, we have that $C(0)\neq 0$, and thus $C(x)^{m-r}$ divides $B(x)$.  But $m-r\ge d$ and so the degree of 
$C(x)^{m-r}$ is strictly greater than the degree of $B(x)$ which implies that $B(x)$ is identically zero.   
Thus we have
$$
\sum_{i=0}^r P_i(x)A^{(i)}(x) \equiv 0~\bmod ~p \, .
$$
Notice that the largest power of $C(x)$ that divides $A^{(i)}(x)$ modulo $p$ is larger than the power 
dividing $A^{(r)}(x)$ modulo $p$ for $i<r$. 
Hence $C(x)$ divides $P_{r}(x)$ modulo $p$.   Taking $Q(x)=P_r(x)$, we get the desired result.
\end{proof}

\begin{cor} Let $f_1(x),\ldots ,f_s(x)\in 
\mathcal{L}$.     Given a prime $p$ and an integer $i$ with $1\leq i \leq s$, let $A_{i}(x)\in \mathbb{Z}[x]$ be such that 
$f_i(x) \equiv A_{i}(x)f_i(x^p)~(\bmod ~p)$. 
Assume that for every $p$ in an infinite set of primes ${\mathcal S}$, 
there are integers $a_1,\ldots ,a_s\in \mathbb{Z}$, not all zero, 
such that the following hold.
\begin{itemize}

\medskip

\item[(i)]  There are two relatively prime polynomials $A(x)$ and $B(x)$ in $\mathbb F_p[x]$ such that 
$A_{1}(x)^{a_1}\cdots A_{s}(x)^{a_s}\equiv \left (A(x)/B(x)\right)^{p-1}~\bmod ~p$.

\medskip

\item[(ii)]  $|a_1|+\cdots +|a_s|\leq \sqrt{p}s$.

\end{itemize}

\medskip 

\noindent 
Then there is a nonzero polynomial $T(x)\in\mathbb Z[x]$ that does not depend on $p$ and 
such that every non-constant irreducible factor of either $A(x)$ or $B(x)$ 
must be a divisor of $T_{\vert p}(x)$ for every $p\in \mathcal S$ large enough. 
\label{cor: 2}
\end{cor}

\begin{proof} 
Let $p$ be in $\mathcal S$. 
Let $C(x)$ be some non-constant irreducible factor of either $A(x)$ or $B(x)$.    
Let $\nu$ denote the valuation on $\mathbb F_p(x)$ induced by $C(x)$.   Then we infer from (i) that 
$$|\nu(A_{1}(x)^{a_1}\cdots A_{s}(x)^{a_s})| \ge p-1\, .$$
But Lemma \ref{lem: 2} gives that there is some  $Q_i(x)\in\mathbb Z[x]$ and some natural number $m_i$ 
(both independent of $p$) such that 
$$
|\nu(A_{i}(x)^{a_i})| \leq |a_i|m_i \, ,
$$ 
unless $C(x)$ divides $Q_{i\vert p}(x)$.   
We then infer from (ii) that $C(x)$ should divide $Q_{i \vert p}(x)$ as soon as $p$ is large enough. 
Then, for $p$ large enough in $\mathcal S$, every irreducible factor of either $A(x)$ or $B(x)$  must divide $T_{\vert p}(x)$, where  
$T(x) := Q_1(x)\cdots Q_s(x)$.  This ends the proof. 
\end{proof}

\begin{lem}
Let $s$ be a natural number and let $f_1(x),\ldots ,f_s(x)\in 
\mathcal{L}$.   Suppose that for infinitely many primes $p$ there 
is a polynomial $Q(x,x_1,\ldots, x_s)\in \mathbb{Z}[x][x_1,\ldots ,x_s]$ of total degree at most 
$\sqrt{p}s$ in $x_1,\ldots ,x_d$ such that  $Q_{\vert p}$ is nonzero and $Q(f_1(x),\ldots ,f_s(x))\equiv 0~(\bmod ~p)$.   
Then there is a nontrivial $\mathbb Q$-linear combination of $f_1'(x)/f_1(x),\ldots, f_s'(x)/f_s(x)$  that belongs to 
$\mathbb Q(x)$.
\label{lem: r}
\end{lem}

\begin{proof} Let $\mathcal{S}$ denote an infinite set of primes for which the assumption of the lemma is satisfied  
and let $p\in \mathcal{S}$ with the property that $p>\sqrt p s$. Then we choose a polynomial 
$Q(x,x_1,\ldots ,x_s)\in \mathbb{Z}[x][x_1,\ldots ,x_s]$ of total degree at most 
$\sqrt{p}s$ in $x_1,\ldots ,x_d$, such that  $Q_{\vert p}$ is nonzero and $Q(f_1(x),\ldots ,f_s(x))\equiv 0~(\bmod ~p)$.   
In addition, we choose $Q$ having, among such polynomials, the fewest number of monomials in $x_1,\ldots ,x_s$ occurring with 
a nonzero coefficient (coefficients are polynomials in $x$).  
As before, we let $A_i(x)$ denote an element of 
$\mathbb Z[x]$ such that 
$f_i(x)\equiv A_i(x)f_i(x^p) (\bmod p)$. 
Since $Q(x,f_1(x),\ldots ,f_s(x))\equiv 0~(\bmod ~p)$ we also have 
$Q(x^p,f_1(x^p),\ldots ,f_s(x^p))\equiv Q(x^p, f_1(x)/A_1(x),\ldots ,f_s(x)/A_s(x))\equiv 0~(\bmod ~p)$.  
We let $\mathcal{T}$ be the set of indices $(i_1,\ldots ,i_s)\in \mathbb{N}^s$ such that $x_1^{i_1}\cdots x_s^{i_s}$ 
occurs in $Q$ with a nonzero coefficient.  
Then we have
\begin{equation}
\label{eq: 1}
\sum_{(i_1,\ldots ,i_s)\in \mathcal{T}} c_{i_1,\ldots ,i_s}(x)f_1^{i_1}\cdots f_s^{i_s} \equiv 0~(\bmod ~p)
\end{equation}
and
\begin{equation}
\label{eq: 2}
\sum_{(i_1,\ldots ,i_s)\in \mathcal{T}} c_{i_1,\ldots ,i_s}(x^p)A_1(x)^{-i_1}\cdots A_s(x)^{-i_s} f_1^{i_1}\cdots f_s^{i_s} \equiv 0~(\bmod ~p) \, .
\end{equation}
Pick $(j_1,\ldots ,j_s)\in \mathcal{T}$. Multiplying Equation (\ref{eq: 1}) by $c_{j_1,\ldots, j_s}(x^p)$ 
and Equation (\ref{eq: 2}) by $c_{j_1,\ldots, j_s}(x)A_1(x)^{j_1}\cdots A_s(x)^{j_s} $ and subtracting, 
we obtain a new relation with a smaller number of terms.  
By minimality, this ensures that all coefficients should be congruent to zero mod $p$.  It thus follows that for 
all $(i_1,\ldots ,i_s)\in \mathcal{T}$ we have 
$$c_{i_1,\ldots ,i_s}(x)c_{j_1,\ldots, j_s}(x^p)
\equiv c_{i_1,\ldots ,i_s}(x^p)c_{j_1,\ldots, j_s}(x)A_1(x)^{j_1-i_1}\cdots A_s(x)^{j_s-i_s} ~\bmod ~p \, .
$$
Equivalently, this gives that  
$$
c_{i_1,\ldots ,i_s}(x)^{-(p-1)}c_{j_1,\ldots, j_s}(x)^{p-1}\equiv A_1(x)^{j_1-i_1}\cdots A_s(x)^{j_s-i_s} ~(\bmod ~p) \, .
$$
Since $\mathcal{T}$ has at least two elements, we see that there exist $a_1,\ldots ,a_s\in \mathbb Z$, not all zero and dependent on $p$, 
with $|a_1|+\cdots +|a_s| \leq \sqrt{p}s$ and such that 
$A_1^{a_1}\cdots A_s^{a_s}\equiv (A(x)/B(x))^{p-1}~(\bmod ~p)$, for some relatively prime polynomials $A(x)$ and $B(x)$ in $\mathbb F_p(x)$. 
By Corollary \ref{cor: 2}, there exists a polynomial $T(x)\in \mathbb Z[x]$ that does not depend on $p$ 
and such that  every non-constant irreducible factor of either $A(x)$ or $B(x)$ must be a divisor of $T_{\vert p}(x)$.  
Set $h(x):=f_1^{-a_1}(x)\cdots f_s^{-a_s}(x)$ and $R(x)=A(x)/B(x)$.  Then
\begin{eqnarray*}
h(x^p)&=&f_1^{-a_1}(x^p)\cdots f_s^{-a_s}(x^p) \\ 
&\equiv& f_1^{-a_1}(x)\cdots f_s^{-a_s}(x)A_1^{a_1}(x)\cdots A_s^{a_s}(x) ~(\bmod p)\\ 
&\equiv& h(x)R(x)^{p-1}~(\bmod p)
\end{eqnarray*} 
and so
$h(x)$ is congruent to a scalar multiple of $R(x)$ modulo $p$. Without loss of generality, we may assume that 
$f_1^{-a_1}\cdots f_s^{-a_s}\equiv R(x)~(\bmod ~p)$.
Differentiating with respect to $x$ and dividing by $f_1^{-a_1}\cdots f_s^{-a_s}$, 
we obtain that 
$$
\sum_{i=1}^s a_i f_i'(x)/f_i(x)~\equiv ~R'(x)/R(x)~(\bmod~p)\, .
$$  
Since by assumption $p>\sqrt{p}s$ and not all the $a_i$ are equal to zero, this provides a non-trivial linear 
relation over $\mathbb F_p$. 
Now let us observe that $R'(x)/R(x)\equiv U(x)/T(x) \bmod p$ for some polynomial $U(x)\in \mathbb Z[x]$ of degree less than the degree of 
$T(x)$.  Let $d$ denote the degree of $T(x)$.  Then we just proved that 
$1,x,\ldots, x^{d-1}, T(x)F_1'(x)/F_1(x),\ldots ,T(x)F_s'(x)/F_s(x)$ are $\mathbb F_p$-linearly dependent when reduced modulo $p$.  
Since this holds for infinitely many $p$, 
Lemma \ref{lem: 0} implies the existence of linear relation over $\mathbb{Q}$.  
Dividing such a relation by $T(x)$, we obtain that there exists a nontrivial $\mathbb Q$-linear combination of 
$f_1'(x)/f_1(x),\ldots, f_s'(x)/f_s(x)$  that belongs to 
$\mathbb Q(x)$. This ends the proof.  
\end{proof}

We are now ready to prove Theorem \ref{thm: pA}.

\begin{proof}[Proof of Theorem \ref{thm: pA}] 
Let $s$ and $a$ be two positive integers. Set   
$$
R_s(x_1,\ldots,x_a)\ := \  \sum_{r=6}^{6+s-1} \frac{1}{1-(x_1+\ldots+ x_r)}  \in \mathbb Z[[x_1,\ldots,x_{6+s-1}]] 
$$ and 
$f(x) := \Delta\left(R_s\right)$.   
An easy computation first gives that 
\begin{eqnarray*}
\Delta\left(\displaystyle\frac{1}{1-(x_1+\cdots+x_r)}\right)& =&  \  \sum_{n= 0}^{+\infty} {rn \choose n,\ldots,n} x^n \\ 
  &=& \  \sum_{n= 0}^{+\infty} \frac{(rn)!}{n!^r} x^n =: f_r(x) \, . 
\end{eqnarray*}
Thus, $f(x) = f_6(x) +\cdots + f_{6+s-1}(x)$. 
Our aim is now to prove that $\deg(f_{\vert p}) \geq p^{s/2}$ for every prime $p$ large enough.   

\medskip

We recall that for every $r\geq 1$, the power series $f_r$ belongs to $\mathcal{L}$. We also  
let $A_{r}(x)$ be a polynomial such that $f_r(x) \equiv A_{r}(x)f_r(x^p) \bmod ~p$.  
Notice that Stirling's formula gives
\begin{equation}\label{eq: stirling}
\frac{(rn)!}{n!^r}\sim r^{rn+1/2}\sqrt{2\pi n}^{1-r}
\end{equation} 
and so the radius of convergence of $f_r(x)$ is $1/r^r$ 
and by Pringsheim's theorem, a singularity occurs at $x=1/r^r$.  
This implies that $f_1,f_{2},\ldots$ are linearly independent over $\mathbb Q$.  Indeed, if, for some positive integer $n$, there would 
be a nontrivial relation $a_1f_1 + \cdots a_n f_n=0$ with $a_n\not = 0$. We would have that $a_nf_n = a_1f_1 + \cdots + a_{n-1}f_{n-1}$ 
but the right-hand side is analytic in a neighbourhood of $1/n^n$ while the left-hand side is not. 

\medskip

We also infer from (\ref{eq: stirling}) that 
$$
\frac{(rn)!}{n!^r}r^{-rn}\sim r^{1/2}\sqrt{2\pi n}^{1-r}
\;\;\mbox{  and }\;\;
n\frac{(rn)!}{n!^r}r^{-rn}\sim nr^{1/2}\sqrt{2\pi n}^{1-r}\, .
$$ 
For $r \geq 6$, the right-hand sides are both in ${\rm O}(1/n^{3/2})$ which implies that  
$\lim_{x\to (1/r^r)^{-}} f_r(x)$ and $\lim_{x\to 1/(r^r)^{-}} f'_r(x)$ both exist and are finite. 

\medskip

We observe now that $f_r'(x)/f_r(x)$ must have a singularity at $x=1/r^r$, while it is clearly analytic inside the 
disc of radius $1/r^{r}$. 
Indeed, otherwise it would be analytic in an open ball 
$U$ containing $1/r^r$ and we may define a function $T_r(z)$, analytic in $U$, by declaring
$T_r(z)=\int_{\gamma} f_r'(z)/f_r(z) \, dz$, where $\gamma$ is any path in $U$ from $1/r^r$ to $z$.
Then notice that $f_r(z)\exp(-T_r(z))$ has derivative zero on $U$ and thus
$f_r(z)=C\exp(-T_r(z))$ would be analytic in $U$, contradicting the fact that $f_r$ has a singularity at $1/r^r$. 
Furthermore,  if $t>r$, then $f_r(1/a^a)>0$ is nonzero and hence $f_r'(z)/f_r(z)$ is analytic in some neighbourhood 
of $1/t^t$.   

\medskip

We claim there does not exist a nontrivial $\mathbb Q$-linear combination of elements of 
$\{f_r'(x)/f_r(x)~\mid~r\ge 6\}$ that is equal to a rational function.  To see this, suppose that we have a nontrivial relation
\begin{equation}
\label{eq: xxx}
\sum_{i=6}^n c_i f_{i}'(x)/f_{i}(x) \ = \ R(x) \in \mathbb Q(x)
\end{equation} with $c_n\neq 0$.  Recall that 
$c_n f_{n}'(x)/f_{n}(x)$ has a singularity at $x=1/n^{n}$.  But, by the preceding remarks, the other terms are analytic in a 
neighbourhood of $x=1/n^{n}$ and thus $R(x)$ must have a pole at $x=1/n^n$.  Otherwise, we could express $c_n f_n'(x)/f_n(x)$ 
as a linear combination of elements that are analytic in a neighbourhood of $1/n^n$, which would give a contradiction.  
But if we look at the limit as $x\to 1/n^{n}$ from the left along the real line, we have that the limit of the left-hand side of 
Equation (\ref{eq: xxx}) is a real number since $\lim_{x\to (1/n^n)^{-}} f_n(x)$ and $\lim_{x\to 1/(n^n)^{-}} f'_n(x)$ both exist and are finite, 
while the limit on the right-hand side goes to infinity since $R$ has a pole at $x=1/n^n$, a contradiction.  Thus, there does not exist a nontrivial 
$\mathbb Q$-linear 
combination of elements of $\{f_r'(x)/f_r(x)~\mid~a\ge 6\}$ 
that is equal to a rational function. 

\medskip

Now, by Lemma \ref{lem: r}, we obtain that there do not exist infinitely many primes $p$ for which  
there is a polynomial $Q(x,x_1,\ldots, x_s)\in \mathbb{Z}[x][x_1,\ldots ,x_s]$ of total degree at most 
$\sqrt{p}s$ in $x_1,\ldots ,x_d$ such that  $Q_{\vert p}$ is nonzero and $Q(f_6(x),\ldots ,f_{6+s-1}(x))\equiv 0~(\bmod ~p)$.   
Since $f_6,\ldots, f_{6+s-1}$ are linearly independent over $\mathbb Q$, Lemma \ref{lem: 1} implies that 
$\deg(f_{\vert p} )\geq p^{s/2}$ for every prime $p$ large enough, concluding the proof.
\end{proof}

%%%%%%%%%%%%%%%%%%%%%%%%%%%%%%%%%%%%%%%%%%%%
\section{Connection with enumerative combinatorics, automata theory and decidability}\label{sec: decid}

Formal power series with integer coefficients naturally occur as generating functions in enumerative combinatorics 
(see  \cite{St78,St99}). In this 
area we have the following natural hierarchy:
$$
\left\{\mbox{ rational  }\right\}\; \subset \;\left\{\mbox{ algebraic   }\right\} \;\subset \; \left\{\mbox{ $D$-finite power series }\right\}
$$
where a power series is differentially finite, or $D$-finite for short, if it satisfies a homogeneous linear 
differential equation with polynomial coefficients. Most of the generating functions that are studied in enumerative combinatorics 
turn out to be $D$-finite (see for instance \cite{St99}).  
Now it may be relevant to make the following observation: 
a $D$-finite power series in $\mathbb Z[[x]]$ with a positive radius of convergence is a $G$-function and, according to a conjecture of Christol, 
it should be the diagonal of a rational function, as it is {``}globally bounded" (see \cite{Ch86,An89}). Thus, at least conjecturally, most of 
$D$-finite power series that appear in enumerative combinatorics should be diagonals of rational functions.

\medskip

The present work has some connection with the classical problem of finding congruence relations satisfied by the coefficients of 
generating functions.  Given a generating function $f(x) = \sum_{n=0}^{+\infty} a(n)x^n\in\mathbb Z[[x]]$, a prime number $p$, and two 
nonnegative integers $b$ and $r$, a  standard problem is to determine 
 the integers $n$ such that $a(n)\equiv b \bmod p^r$. 
In other words, the aim is to describe sets such as 
${\mathcal S} := \left\{ n\in \mathbb N \mid  a(n)\equiv b \bmod p^r \right\}$. Now, if $f$ is a diagonal 
of a rational function (which as just explained should be the typical situation), 
the Furstenberg--Deligne theorem implies that $f_{\vert p}$ is algebraic over $\mathbb F_p(x)$. Then  
a classical theorem of Christol \cite{CKMR}, as revisited in \cite{DL}, implies that  the sequence $(a(n)\bmod p^r)$ is $p$-automatic which means that it can be generated by a finite 
$p$-automaton.  In particular, ${\mathcal S}$ is a $p$-automatic set. We recall that an infinite sequence $a$ with values in a finite set 
is said to be $p$-automatic if $a(n)$ is a finite-state function of the base-$p$ 
representation of $n$. Roughly, this means that there exists a finite automaton taking the
 base-$p$ expansion of $n$ as input and producing the term $a(n)$ as 
output. A set ${\mathcal E}\subset \mathbb N$ is said to be  
$p$-automatic if there exists a finite automaton that reads as input the base-$p$ expansion of 
$n$ and accepts this integer (producing as output the symbol $1$) if  $n$ belongs to 
${\mathcal E}$, otherwise this automaton rejects the integer 
$n$, producing as output the symbol $0$. For more formal definitions we refer the reader to \cite{AdBe}. 
 It was already noticed in \cite{LvdP} that 
such a description in terms of automata provides a vast range of congruences for coefficients of diagonals of rational power series.  
The present work emphasizes some effective aspects related to these congruences.  
Indeed, by Theorem \ref{theo: modp} we are able to give an effective 
bound for the degree and the height of the algebraic function $f_{\vert p^r}$. 
As explained in \cite{AdBe}, this allows to bound the number of states of the underlying 
$p$-automaton. This gives the following result.

\begin{thm} 
Let $f(x) = \sum_{n=0}^{+\infty} a(n)x^n\in\mathbb Z[[x]]$ be the diagonal of an algebraic function. Let $b$ and $r$ be positive integers and  
$p$ be a prime number. Then the set 
$$
{\mathcal S} := \left\{ n\in \mathbb N \mid  a(n)\equiv b \bmod p^r \right\}$$
is a $p$-automatic set that can be effectively determined. 
In particular, the following properties are all decidable:

\medskip

\begin{itemize}
\item[{\rm (i)}] the set ${\mathcal S}$ is empty.

\medskip

\item[{\rm (ii)}] the set ${\mathcal S}$ is finite.

\medskip

\item[{\rm (iii)}] the set ${\mathcal S}$ is periodic, that is, formed by the union of a finite set and 
of a finite number of  arithmetic progressions.
\end{itemize}
\end{thm}

\medskip 

\begin{figure}[h]
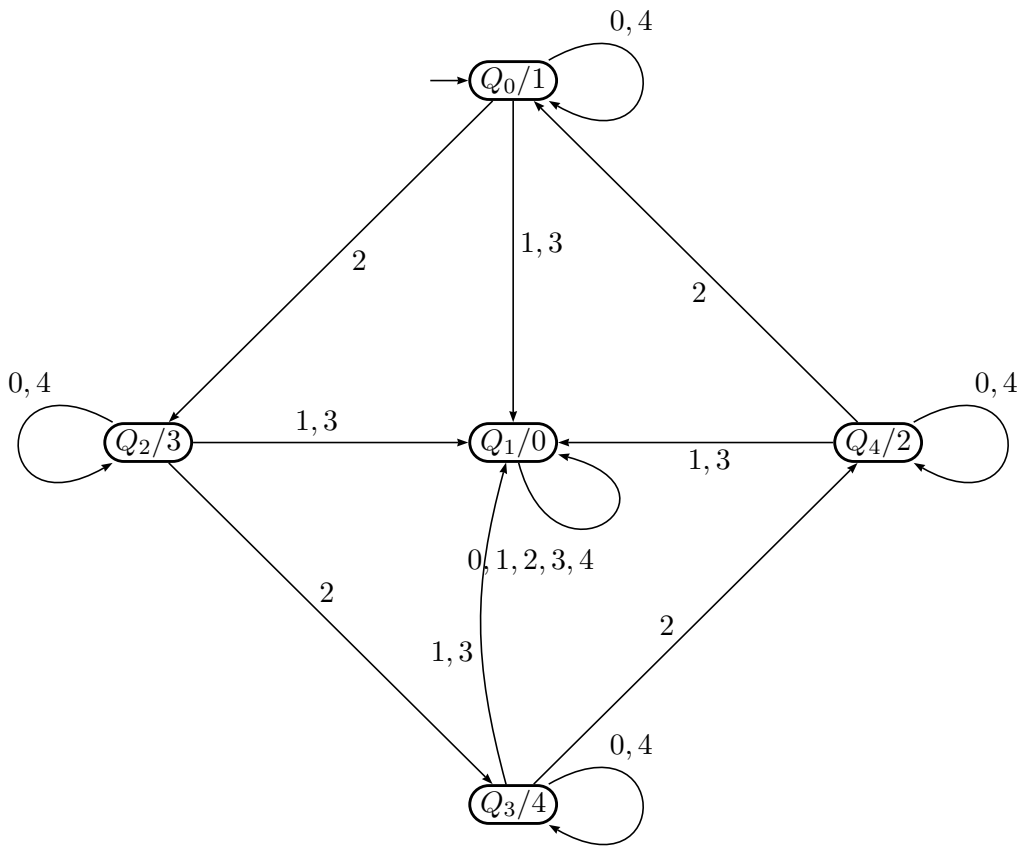


\centering 
\VCDraw{  \begin{VCPicture}{(0,-18)(16,1)}
    % states
    \StateVar[Q_0/1]{(8,0)}{A} \StateVar[Q_2/3]{(0,-8)}{B}   \StateVar[Q_1/0]{(8,-8)}{C}    \StateVar[Q_4/2]{(16,-8)}{D} 
     \StateVar[Q_3/4]{(8,-16)}{E}   
    % initial--final
    \Initial{A} 
    % transitions
     \EdgeL{A}{B}{2}   \EdgeL{B}{E}{2} \EdgeL{E}{D}{2} \EdgeL{D}{A}{2}
     
    \EdgeL{A}{C}{1,3}   \EdgeL{B}{C}{1,3} \EdgeL{D}{C}{1,3}  \ArcL{E}{C}{1,3}
   
     \LoopSE{C}{\;\;\;\;\;\;\;\;\;\; 0,1,2,3,4}    \LoopE{E}{0,4}  \LoopE{A}{0,4} \LoopW{B}{0,4}  \LoopE{D}{0,4}

  \end{VCPicture}}
  \caption{A $5$-automaton generating the Ap\'ery sequence modulo $5$.}
  \label{AB:figure:ap}
\end{figure}

As an illustration, we give in Figure \ref{AB:figure:ap} the picture of a $5$-automaton that generates the Ap\'ery 
numbers $a(n) = \sum_{k=0}^n {n\choose k}^2{n+k \choose k}^2$ modulo $5$.  
We thus have that:  
$a(n) \equiv  0 \bmod 5$ if the base-$5$ expansion of $n$ contains at least a $1$ or a $3$;  
$a(n) \equiv  1 \bmod 5$ if the base-$5$ expansion of $n$ 
does not contains the digits $1$ and $3$ and if the number of $2$'s is congruent to $0\bmod 4$;  
 $a(n) \equiv  2 \bmod 5$ if the base-$5$ expansion of $n$ 
does not contains the digits $1$ and $3$ and if the number of $2$'s is congruent to $3\bmod 4$; 
 $a(n) \equiv  3 \bmod 5$ if the base-$5$ expansion of $n$ 
does not contains the digits $1$ and $3$ and if the number of $2$'s is congruent to $1\bmod 4$; 
 $a(n) \equiv  4 \bmod 5$ if the base-$5$ expansion of $n$ 
does not contains the digits $1$ and $3$ and if the number of $2$'s is congruent to $2\bmod 4$. 
In this direction, Beukers made the following conjecture \cite{Be86}: if $r$ denotes the sum of the number of $1$'s and the number 
of $3$'s in the base-$5$ expansion of $n$, then $a(n)\equiv 0 \bmod p^r$. 
Recently, Delaygue \cite{Delaygue} announced a proof of this conjecture. 
In order to answer this kind of question, it would be interesting to understand, 
given the diagonal of a rational power series 
$f(x) = \sum_{n\geq 0} a(n)x^n\in\mathbb Z[[x]]$,  the connection between the $p$-automaton that generates 
$a(n)\bmod p^r$ and the one that generates $a(n) \bmod p^{r+1}$ for every positive integer $r$.

%%%%%%%%%%%%%%%%%%%%%%%%%%%%%%%%%%%
\section{Algebraic independence for $G$-functions with the Lucas property}\label{sec: alg}

In this section, we come back to the results  obtained in Section \ref{sec: high} when proving Theorem \ref{thm: pA}. 
It turns out that 
we have incidentally  proved a result of independent interest about  algebraic independence of some $G$-functions. 

\medskip

We recall 
that a sequence $a:\mathbb{N}\to \mathbb{Z}$ has the Lucas property if for every prime $p$ we have $a(pn+j)\equiv a(n)a(j)\, (\bmod ~p)$.   
In 1980, Stanley \cite{St80} conjectured that, for positive integer $t$, the power series $\sum_{n= 0}^{+\infty} {2n \choose n}^t x^n$ 
is transcendental over $Q(x)$ unless $t=1$, in which case it is equal to $1/\sqrt{ 1 - 4x}$. He also proved the transcendence in the case where 
 $t$ is even. The conjecture was proved independently by Flajolet \cite{Fl} and by Sharif and Woodcock \cite{SW2} with two different methods. The proof of 
Sharif and Woodcock is based on the Furstenberg--Deligne theorem and use the fact that the sequence ${2n \choose n}^t $ satisfies the Lucas 
property. These authors also proved in the same way the transcendence of $ \sum_{n= 0}^{+\infty} {rn \choose n,\ldots,n}^t x^n$ 
for every integers $r\geq 3$, $t\geq 1$.  Their approach was then developed by Allouche {\it et al.} in \cite{AGS} (see also \cite{All99}) 
who obtained a general criterion for the algebraicity of formal power series with coefficients in $\mathbb Q$ satisfying the Lucas property. 
However, it seems that not much is known about algebraic independence of such power series. 
As a first result in this direction, we prove Theorem \ref{cor: 3} below.   
We recall that $\mathcal{L}$ denotes the set of all 
power series in $\mathbb{Z}[[x]]$ that have constant coefficient one, whose sequence of coefficients has the Lucas property, and that 
satisfy a homogeneous linear differential equation with coefficients in $\mathbb Q(x)$.

\begin{thm} Let $f_1,\ldots ,f_s$ be elements of $\mathcal{L}$ such that 
there is no nontrivial  $\mathbb Q$-linear combination of $f_1'/f_1,\ldots ,f_s'/f_s$  that belongs to $\mathbb Q(x)$. 
Then $f_1,\ldots ,f_s$ are algebraically independent over $\mathbb{Q}(x)$.
\label{cor: 3}
\end{thm}

\begin{proof} Let us assume that $f_1,\ldots, f_s$ are algebraically dependent. Then there exists a nonzero polynomial 
$Q\in\mathbb Z[x,x_1,\ldots,x_s]$ such that $Q(x,f_1,\ldots ,f_s)=0$. Note that for all sufficiently large primes $p$, the total degree of 
$Q$ is less than $\sqrt p s$ and $Q_{\vert p}$ is nonzero. 
Thus Lemma \ref{lem: r} implies the existence of nontrivial  $\mathbb Q$-linear combination of $f_1'/f_1,\ldots ,f_s'/f_s$ 
that is equal to a rational function, a contradiction. 
\end{proof}

We then deduce the following consequences of Theorem \ref{cor: 3}. 

\begin{cor}\label{cor: ai}
Set  $\displaystyle f_r(x) := \sum_{n= 0}^{+\infty} {rn \choose n,\ldots,n} x^n$ and $\displaystyle g_r := \sum_{n= 0}^{+\infty} {2n \choose n}^r x^n$. 
Then  $\left\{ f_r \mid r\geq 6\right\}$ and $\left\{ g_r \mid r\geq 4\right\}$
are two families of algebraically independent functions over $\mathbb Q(x)$. 
\end{cor}

\begin{proof}
The fact that $f_r$ and $g_r$ belong to $\mathcal L$ and are transcendental over $\mathbb Q(x)$ can be found in \cite{SW2}. Furthermore, 
we already obtained in the proof of Theorem \ref{thm: pA} in Section \ref{sec: high} 
that there is no $\mathbb Q$-linear combination of  $f_r'/f_r,\ldots ,f_{r+n}'/f_{r+n}$, $r\geq 6$, $s\geq 1$  
that is equal to a rational function. 
The fact that there is no $\mathbb Q$-linear combination of  $g_r'/g_r,\ldots ,g_{r+n}'/g_{r+n}$, $r\geq 4$, $s\geq 1$  
that is equal to a rational function can be proved in a very similar way.  
Thus Theorem \ref{cor: 3} applies, which implies the result. 
\end{proof}

We note that Theorem \ref{cor: 3} can actually be used to prove the best possible results regarding algebraic independence of 
both families considered in Corollary \ref{cor: ai}. Indeed, we could obtain that 
$\left\{ f_r \mid r\geq 3\right\}$ and $\left\{ g_r \mid r\geq 2\right\}$
are two families of algebraically independent functions over $\mathbb Q(x)$.  
We choose to only give the statement in Corollary \ref{cor: ai} here, as   
it is a direct consequence of the results already proved in Section \ref{sec: high} and does not need additional work. 
Furthermore, it may be the case that Theorem \ref{cor: 3} also has interesting applications regarding algebraic independence 
of other classical families of $G$-functions. Since it is not the focus of the present paper, we plan to investigate 
this question  in more detail  in a future work. 
%%%%%%%%%%%%%%%%%%%%%%%%%%%%%%%%%%%%%%%%%%%%%%

\end{document}